\documentclass[11pt]{article}
\usepackage{enumerate}
\usepackage{amssymb,a4wide,latexsym,makeidx,epsfig}
\usepackage{amsthm}
\usepackage{dsfont}
\usepackage{amsmath}
\usepackage{lipsum}
\usepackage{enumerate}
\usepackage{mathrsfs}
\usepackage{xcolor}
\usepackage{tikz}
\usepackage{setspace}
\usepackage{geometry}
\usepackage{appendix}
\usepackage{multirow}
\usepackage{subcaption}
\usepackage[colorlinks=true, linkcolor=black, anchorcolor=black, citecolor=black, urlcolor=black, CJKbookmarks=true]{hyperref}
\usepackage{microtype}
\usepackage{colonequals}
\usepackage{enumitem}
\allowdisplaybreaks

\newtheorem{theorem}{Theorem}[section]
\newtheorem{remark}{Remark}[section]

\newtheorem{definition}[theorem]{Definition}
\newtheorem{lemma}[theorem]{Lemma}
\newtheorem{fact}{Fact}[section]

\newtheorem{corollary}[theorem]{Corollary}

\newtheorem{claim}{Claim}[section]

\numberwithin{equation}{section}

\begin{document}
\textwidth 150mm \textheight 225mm

\title{On the rainbow Cameron-Erd\H{o}s problem with respect to generalized Sidon sets of multidimensional grids
}

\author{
Xihe Li\footnote{School of Mathematics and Statistics, Shaanxi Normal University, Xi'an, Shaanxi 710119, China.}~\footnote{Corresponding author.}~~~~~~
Runshan Wang\footnotemark[1]~~~~~~
}
\date{}
\maketitle
\newcommand\blfootnote[1]{%
\begingroup
\renewcommand\thefootnote{}\footnote{#1}%
\addtocounter{footnote}{-1}%
\endgroup
}
\blfootnote{E-mail addresses: xiheli@snnu.edu.cn, rswangmath@163.com.}
\begin{center}
\begin{minipage}{120mm}
\begin{center}
{\small {\bf Abstract}}
\end{center}
{\small For positive integers $n$, $d$, $k$ and $h$, let $[n]^d$ be the $d$-dimensional grid of order $n$, and we refer to the equation $\sum_{i=1}^{h}x_{1,i}=\sum_{i=1}^{h}x_{2,i}=\cdots =\sum_{i=1}^{h}x_{k,i}$ as the {\it $B_{k,h}$-equation}, where $x_{\ell,i}$ ($\ell \in [k]$, $i\in [h]$) are $kh$ points in $[n]^d$.
In this paper, we study the rainbow Cameron-Erd\H{o}s problem with respect to the $B_{k,h}$-equation.
We obtain the asymptotic number of $r$-colorings of $[n]^d$ without rainbow solutions to the $B_{k,h}$-equation, and we show that the typical colorings with this property are $(kh-1)$-colorings.
We also prove that among all subsets of $[n]^d$, $[n]^d$ is the unique subset admitting the maximum number of $r$-colorings without rainbow solutions to the $B_{k,h}$-equation.
The case $d=1$ and $k=h=2$ of our result confirms a conjecture on Sidon sets by Lin, Wang and Zhou~[{\it European J. Combin.}, 2022];
the case $d=1$, $k=2$ and $h\geq 2$ of our result partly solves a problem concerning linear equations proposed by Cheng, Jing, Li, Wang and Zhou~[{\it J. Combin. Theory Ser. A}, 2023];
the case $d\geq 2$ and $k=h=2$ corresponds to colorings without rainbow (possibly degenerate) parallelograms, and this geometric perspective might be of independent interest.
Our proof combines the hypergraph container method with a stability analysis and a deviation gain argument.
\hspace{2em}

\vskip 0.1in \noindent {\bf AMS Subject Classification (2020)}: \ 05A16, 11B30, 11B75
\vskip 0.1in \noindent {\bf Keywords}: \ Counting and typical structures, rainbow Cameron-Erd\H{o}s problem, Sidon sets, hypergraph container method, colorings
}
\end{minipage}
\end{center}

\section{Introduction}
\label{sec:introduction}

One of the fundamental topics in extremal combinatorics is counting discrete structures that possess certain properties and analyzing their typical structures.
A major problem in this field, dating back to 1990 and initiated by Cameron and Erd\H{o}s \cite{CaEr}, is to determine the number of subsets of positive integers that satisfy a given constraint.
For the number of sum-free subsets, we refer the reader to \cite{ABMS,BLST18,Gho,Gre,LiSh,Sap};
for results on $B_h$-sets, Sidon sets or multiplicative Sidon sets, we refer to~\cite{DKLRS16,DKLRS18,KLRS,LiPa,SaTh16};
and for the number of subsets without $k$-term arithmetic progressions, we refer to~\cite{BaLS,BHMNW} .

In the setting of colored discrete structures, Erd\H{o}s and Rothschild \cite{Erd74} asked which graphs on $n$ vertices admit the maximum number of $r$-edge-colorings without a monochromatic subgraph $H$.
The problem of Erd\H{o}s and Rothschild was also generalized to other structures.
For example, Hoppen, Kohayakawa and Lefmann~\cite{HoKL} obtained an Erd\H{o}s-Ko-Rado type result for monochromatic intersecting systems; Liu, Sharifzadeh and Staden~\cite{LiSS} and H\`{a}n and Jim\'{e}nez~\cite{HaJi} determined the maximum number of monochromatic sum-free colorings of integers and of finite abelian groups, respectively.
In addition to monochromatic substructures, this problem was also extended to colorings with other forbidden color patterns, such as rainbow substructures.
For instance, the number of rainbow triangle-free $r$-edge-colorings of complete graphs (also known as Gallai colorings) was determined in \cite{BaLi,BaBH};
Li, Broersma and Wang~\cite{LiBW21} obtained the asymptotic number of $r$-colorings of $[n]$ or $\mathbb{Z}_n$ without rainbow 3-term arithmetic progressions, which was extended to $k$-term arithmetic progressions by Lin, Wang and Zhou~\cite{LiWZ};
Cheng, Jing, Li, Wang and Zhou~\cite{CJLWZ} studied the number of rainbow sum-free colorings of integers and their typical structures.

In this paper, we extend the above problems to study rainbow solutions to generalized Sidon-equations in multidimensional grids under coordinatewise addition.
For positive integers $n$ and $d$, let $[n]\colonequals \{1, 2, \ldots, n\}$, and let $[n]^d$ be the {\it $d$-dimensional grid of order $n$}.
Note that $[n]^d$ has $n^d$ points, and each point can be viewed as a $d$-dimensional vector in which each coordinate is an integer within the set $[n]$.
For integers $k\geq 2$ and $h\geq 2$, we refer to the equation $$\sum_{i=1}^{h}x_{1,i}=\sum_{i=1}^{h}x_{2,i}=\cdots =\sum_{i=1}^{h}x_{k,i}$$ as the {\it $B_{k,h}$-equation},
where $x_{\ell,i}$ ($\ell \in [k]$, $i\in [h]$) are $kh$ points in $[n]^d$.
In particular, when $d=1$ and $k=h=2$, the equation is usually referred to as the {\it Sidon equation}, and a set $A\subseteq [n]$ is called a {\it Sidon set} if all sums $x_1+x_2$ with $x_1, x_2\in A$ are distinct (i.e., $A$ contains no nontrivial solutions to the Sidon equation);
the case that $d=1$, $k=2$ and $h\geq 2$ corresponds to the {\it $B_h$-set}, which is defined as a set $A\subseteq [n]$ such that all sums $x_1+\cdots +x_h$ with $x_1, \ldots, x_h\in A$ are distinct.

For an integer $r\geq 1$ and a subset $A\subseteq [n]^d$, let $c\colon\, A \to [r]$ be an {\it $r$-coloring} of $A$.
Given an $r$-coloring $c$ of $A$, a solution $x_{1,1}, \ldots, x_{1,h}, x_{2,1}, \ldots, x_{2,h}, \ldots, x_{k,1}, \ldots, x_{k,h}$ to the $B_{k,h}$-equation in $A$ is called {\it rainbow}
if the colors $c(x_{\ell,i})$ ($\ell \in [k]$, $i\in [h]$) are pairwise distinct.
For any $A\subseteq [n]^d$ and $r\geq kh$, we use $g_r(A)$ to denote the number of $r$-colorings of $A$ without rainbow solutions to the $B_{k,h}$-equation.
By considering colorings of $A$ with exactly $kh-1$ colors (note that a $(kh-1)$-coloring certainly contains no rainbow solutions to the $B_{k,h}$-equation), a lower bound on $g_r(A)$ is
\begin{align}\label{eq:lower}
g_r(A)\geq &~\binom{r}{kh-1}\sum_{i=0}^{kh-2}(-1)^i {kh-1\choose i} (kh-1-i)^{|A|} \nonumber\\
= &~(1-o(1))\binom{r}{kh-1}(kh-1)^{|A|}.
\end{align}
Our main result shows that this lower bound is asymptotically sharp when $A$ is dense.

\begin{theorem}\label{th:Bkh}
For integers $d, k, h, r, n$ with $d\geq 1$, $k\geq 2$, $h\geq 2$, $r\geq kh$ and $n$ sufficiently large,
if $A\subseteq [n]^d$ is a subset with $|A|\geq n^d-\frac{n^d}{\log n}$, then
$$g_r(A)= (1+o(1))\binom{r}{kh-1}(kh-1)^{|A|}.$$
Moreover, all but a $o(1)$ proportion of $r$-colorings of $A$ without rainbow solutions to the $B_{k,h}$-equation are $(kh-1)$-colorings.
Furthermore, among all subsets of $[n]^d$, $[n]^d$ is the unique subset admitting the maximum number of $r$-colorings without rainbow solutions to the $B_{k,h}$-equation.
\end{theorem}

For colorings of the grid $[n]^d$, we can deduce the following consequence immediately.

\noindent\begin{corollary}\label{co:Bh}
For integers $d, k, h, r, n$ with $d\geq 1$, $k\geq 2$, $h\geq 2$, $r\geq kh$ and $n$ sufficiently large,
we have
$$g_r([n]^d)= (1+o(1))\binom{r}{kh-1}(kh-1)^{n^d}.$$
Moreover, all but a $o(1)$ proportion of $r$-colorings of $[n]^d$ without rainbow solutions to the $B_{k,h}$-equation are $(kh-1)$-colorings.
\end{corollary}

We now provide three remarks regarding our result.

\begin{remark}\label{re:Sidon} ~
{\rm
\begin{itemize}
\item[(1)] In 2022, Lin, Wang and Zhou~\cite{LiWZ} conjectured that for sufficiently large $n$, $[n]$ is the unique subset admitting the maximum number of $r$-colorings without rainbow solutions to $x_1+x_2=y_1+y_2$ among all subsets of $[n]$, and almost all such $r$-colorings of $[n]$ use at most three colors (see \cite[Conjecture~6.1]{LiWZ}). The special case $d=1$ and $k=h=2$ of our result confirms this conjecture.

\item[(2)] In 2023, Cheng, Jing, Li, Wang and Zhou~\cite{CJLWZ} proposed the following problem: for pairs $(h_1, h_2)$ of integers, study the number of $r$-colorings of $[n]$ without rainbow solutions to the equation $\sum_{i=1}^{h_1}x_i=\sum_{i=1}^{h_2}y_i$ (see \cite[Section~6]{CJLWZ}). The special case $d=1$, $k=2$ and $h\geq 2$ of our result solves this problem for the case $h_1=h_2=h$. For the case $h_1\neq h_2$, we shall provide more remarks in Section~\ref{sec:conclu}.

\item[(3)] From a geometric perspective, when $d\geq 2$, $k=2^{t-1}$ and $h=2$ for some $t \geq 2$, the equation $x_{1,1}+x_{1,2}=\cdots=x_{2^{t-1},1}+x_{2^{t-1},2}$ describes a (possibly degenerate) $t$-dimensional parallelotope in the grid $[n]^d$. In particular, in the case $d\geq 2$ and $k=h=2$, the equation $x_1+x_2=y_1+y_2$ describes a (possibly degenerate) parallelogram. Hence, we can also view this case of our result as a counting and typical structure result for colorings without rainbow parallelograms. For nondegenerate (i.e., the four points are not collinear) parallelograms, the result is still true, and we will discuss this in Section~\ref{sec:conclu}.
\end{itemize}
}
\end{remark}

Let $\mathbb{Z}_n$ be the cyclic group of order $n$ formed by the set $\{0, 1, \ldots, n-1\}$ with the binary operation addition modulo $n$.
For any point $v\in [n]^d$ or $v\in \mathbb{Z}_n^d$, let $v(1), \ldots, v(d)$ be its coordinates, i.e., $v=(v(1), \ldots, v(d))$.
A solution to the $B_{k,h}$-equation in $\mathbb{Z}_n^d$ consists of $kh$ points $x_{\ell,i}$ ($\ell \in [k]$, $i\in [h]$) in $\mathbb{Z}_n^d$ with
$\sum_{i=1}^{h}x_{1,i}(j)\equiv\sum_{i=1}^{h}x_{2,i}(j)\equiv\cdots\equiv\sum_{i=1}^{h}x_{k,i}(j) \pmod{n}$
for all $j\in [d]$.
For any $A\subseteq \mathbb{Z}_n^d$ and $r\geq kh$, we use $g_r(A; \mathbb{Z}_n^d)$ to denote the number of $r$-colorings of $A$ without rainbow solutions to the $B_{k,h}$-equation.
For any $A\subseteq [n]^d$, if we also view $A$ as a subset of $\mathbb{Z}_n^d$, then $g_r(A; \mathbb{Z}_n^d)\leq g_r(A)$ since a solution to the $B_{k,h}$-equation in $[n]^d$ is also a solution in $\mathbb{Z}_n^d$.
Moreover, the lower bound $g_r(A; \mathbb{Z}_n^d)\geq (1-o(1))\binom{r}{kh-1}(kh-1)^{|A|}$ also holds.
Hence, we have the following corollary (the furthermore part follows from $g_r(A; \mathbb{Z}_n^d)\leq g_r(A)$ and the proof of Theorem~\ref{th:Bkh}).

\begin{corollary}\label{co:Zn}
For integers $d, k, h, r, n$ with $d\geq 1$, $k\geq 2$, $h\geq 2$, $r\geq kh$ and $n$ sufficiently large,
if $A\subseteq \mathbb{Z}_n^d$ is a subset with $|A|\geq n^d-\frac{n^d}{\log n}$, then
$$g_r(A; \mathbb{Z}_n^d)= (1+o(1))\binom{r}{kh-1}(kh-1)^{|A|}.$$
Moreover, all but a $o(1)$ proportion of $r$-colorings of $A$ without rainbow solutions to the $B_{k,h}$-equation are $(kh-1)$-colorings.
Furthermore, among all subsets of $\mathbb{Z}_n^d$, $\mathbb{Z}_n^d$ is the unique subset admitting the maximum number of $r$-colorings without rainbow solutions to the $B_{k,h}$-equation.
\end{corollary}

The remainder of this paper is organized as follows.
In the next section, we supply a proof sketch of our main result.
In Section~\ref{sec:container}, we first present several counting results on solutions to the $B_{k,h}$-equation, and then provide a container theorem with respect to forbidden rainbow solutions to the $B_{k,h}$-equation.
In Section~\ref{sec:sta}, we establish certain stability phenomenon for the templates obtained from the container theorem.
In Section~\ref{sec:proof}, we complete our proof of Theorem~\ref{th:Bkh}.
Finally, we conclude the article with some remarks and open problems in Section~\ref{sec:conclu}.

\section{Proof sketch of Theorem~\ref{th:Bkh}}
\label{sec:sketch}

In this section, we supply a proof sketch of Theorem~\ref{th:Bkh}.
In our proof, expanding on the Hypergraph Container Method developed independently by Balogh-Morris-Samotij~\cite{BaMS} and Saxton-Thomason~\cite{SaTh},
we introduce a stability analysis and a deviation gain argument for counting colorings without rainbow solutions to the $B_{k,h}$-equation.
Our proof consists of four steps as follows:
\begin{itemize}[leftmargin=0.6cm, itemindent=1.1cm]
\item[{\bf Step 1.}] {\bf Establishing the container theorem with respect to forbidden rainbow solutions to the $B_{k,h}$-equation (see Theorem~\ref{th:BkhCT}).}
More precisely, we show that there exists a collection $\mathcal{C}$ of the so-called $r$-templates such that:
(1)~every $r$-coloring without rainbow solutions to the $B_{k,h}$-equation is a subtemplate of some template in $\mathcal{C}$;
(2)~each template in $\mathcal{C}$ contains a small number of rainbow solutions to the $B_{k,h}$-equation;
(3)~the size of $\mathcal{C}$ is small.
It is worth noting that our container theorem not only holds for the grid $[n]^d$, but also for any subset $A\subseteq [n]^d$ with $|A|\geq \frac{n^d}{\log n}$.
This enables us to apply the theorem when analyzing medium sets in Step~4 of our proof.
\vspace{-0.2cm}

\item[{\bf Step 2.}] {\bf Revealing the stability phenomenon for the collection $\mathcal{C}$ of templates given by Step 1.}
We show that $\mathcal{C}$ can be partitioned into two parts $\mathcal{C}_{good}$ and $\mathcal{C}_{bad}$ such that:
(1)~for every template $P\in \mathcal{C}_{good}$, there exists a set $C$ of $kh-1$ common colors so that almost all points are colored by these colors;
(2)~the number of $r$-colorings related to $\mathcal{C}_{bad}$ are sufficiently small.
\vspace{-0.1cm}

\item[{\bf Step 3.}] {\bf Completing the proof for counting and typical structures.}
From Step 2, we know that for templates in $\mathcal{C}_{good}$, almost all points receive a limited set of colors.
The remaining main task is to study the deviation gain, that is, the number of $r$-colorings in which at least one point $v$ receives an unrestricted color.
To achieve this, we first construct a large number of pairwise disjoint (except for the point $v$) solutions containing $v$.
Then, by using the property of forbidden rainbow solutions, we can effectively estimate the deviation gain (see Lemma~\ref{le:dev}).
\vspace{-0.1cm}

\item[{\bf Step 4.}] {\bf Completing the proof for the extremal structure part.}
We show that among all subsets of $[n]^d$, $[n]^d$ itself is the unique subset admitting the maximum number of $r$-colorings without rainbow solutions to the $B_{k,h}$-equation.
To achieve this, we consider three cases based on the density of $A\subseteq [n]^d$: sparse sets ($|A|\leq \frac{n^d}{\log n}$), medium sets ($\frac{n^d}{\log n}\leq |A|\leq n^d-\frac{n^d}{\log n}$) and dense sets ($|A|\geq n^d-\frac{n^d}{\log n}$).
In particular, for medium sets, we also utilize the results obtained from Steps 1 and 2.
\vspace{-0.1cm}
\end{itemize}

\section{Preliminaries}
\label{sec:container}

In this section, we first state and prove several counting results on solutions to the $B_{k,h}$-equation in Section~\ref{subsec:f}.
Then in Section~\ref{subsec:container}, we introduce the hypergraph container method, and establish a container theorem with respect to forbidden rainbow solutions to the $B_{k,h}$-equation.

Throughout this paper, all logarithms are taken with base 2, unless otherwise indicated.
Given a set $A$ and a positive integer $t$, let ${A\choose t}\colonequals \{B\subseteq A\colon\, |B|=t\}$ (i.e., the set of all $t$-element subsets of $A$).
We should remark that since we focus on rainbow solutions in colorings, it suffices to study solutions to the $B_{k,h}$-equation in which $x_{\ell,i}$ ($\ell \in [k]$, $i\in [h]$) are pairwise distinct.
In this paper, whenever we refer to a set, it is implicitly understood that the elements within the set are pairwise distinct and the order of elements is irrelevant;
when discussing a tuple, the elements may not be distinct and the order is relevant.
Moreover, we view the union $X\cup Y$ of two sets as one set, so in particular, the elements within $X\cup Y$ are pairwise distinct and the order of elements is irrelevant.
We also remark that we do not make any serious attempt to optimize absolute constants in our statements and proofs.
For clarity, we denote the constant $c$ in each conclusion by appending the label number of the conclusion as a subscript.
For instance, the constant in Lemma~\ref{le:fdkh} is denoted by $c_{\ref{le:fdkh}}$.

\subsection{Counting solutions to the $B_{k,h}$-equation}
\label{subsec:f}

For any subset $A\subseteq [n]^d$, define $f_{d,k,h}(A)$ as the number of distinct subsets $X$ of $A$ with $|X|=kh$, such that there exists a way to partition the elements of $X$ into $k$ groups of $h$ elements each, and the sum of elements in each group is equal.
In other words, we can partition $X$ into $k$ subsets $\{x_{\ell,1}, \ldots, x_{\ell,h}\}$ ($\ell \in [k]$) so that $\sum_{i=1}^{h}x_{1,i}=\sum_{i=1}^{h}x_{2,i}=\cdots =\sum_{i=1}^{h}x_{k,i}$.
We will use the following lower and upper bounds on $f_{d,k,h}(A)$ in our proofs.

\begin{lemma}\label{le:fdkh}
Let $d, k, h, n$ be integers with $d\geq 1$, $k\geq 2$, $h\geq 2$ and $n$ sufficiently large.
For any subset $A\subseteq [n]^d$ with $|A|\geq n^{\frac{k-1}{k}d}\log n$, there exists a constant $c_{\ref{le:fdkh}}=c_{\ref{le:fdkh}}(d,k,h)>0$ such that
$$c_{\ref{le:fdkh}}\frac{|A|^{kh}}{n^{d(k-1)}} \leq f_{d,k,h}(A) \leq |A|^{k(h-1)+1}.$$
\end{lemma}

\begin{proof}
For the upper bound, note that there are at most ${|A|\choose k(h-1)+1}$ choices of $\{x_{1,i}\colon\, 1\leq i\leq h\}\cup \{x_{\ell,i}\colon\, 2\leq \ell\leq k, 1\leq i\leq h-1\}$,
and after choosing these points, there is at most one choice of $x_{\ell,h}$ satisfying $\sum_{i=1}^{h}x_{1,i}=\sum_{i=1}^{h}x_{\ell,i}$ for each $2\leq \ell\leq k$.
Hence, we have $f_{d,k,h}(A) \leq {|A|\choose k(h-1)+1}\leq |A|^{k(h-1)+1}.$

We next consider the lower bound.
Note that for every $\ell\in [k]$, $\sum_{i=1}^{h}x_{\ell,i}$ is a point in $[hn]^d$.
For any point $s\in [hn]^d$, let $m_A(s)$ be the number of subsets $\{x_1, \ldots, x_h\}\subseteq A$ with $x_1+\cdots +x_h=s$.
Then $\sum_{s\in [hn]^d}m_A(s)={|A|\choose h}\geq \left(\frac{|A|}{h}\right)^h$.
Let $M_1$ be the number of sets $\big\{\{x_{1,1}, \ldots, x_{1,h}\}, \ldots, \{x_{k,1}, \ldots, x_{k,h}\}\big\}$ consisting of $k$ subsets
with $\sum_{i=1}^{h}x_{1,i}=\sum_{i=1}^{h}x_{2,i}=\cdots =\sum_{i=1}^{h}x_{k,i}$, where $x_{\ell,i}$ ($\ell \in [k]$, $i\in [h]$) are $kh$ (not necessarily pairwise distinct) points in $A$.
By Jensen's inequality\footnote{We view ${x\choose k}$ as a convex function on the reals by defining
${x\choose k}\colonequals
\left\{
   \begin{aligned}
    &x(x-1)\cdots(x-k+1)/k!, & & \mbox{if $x\geq k-1$},\\
    &0, & & \mbox{if $x<k-1$}.
   \end{aligned}
   \right.$},
we have
\begin{align}\label{eq:fdkhM1}
  M_1= &~\sum_{s\in [hn]^d} {m_A(s) \choose k}\geq (hn)^d{\frac{1}{(hn)^d}\sum_{s\in [hn]^d}m_A(s) \choose k} \geq (hn)^d{\frac{1}{(hn)^d}\left(\frac{|A|}{h}\right)^h \choose k} \nonumber\\
   \geq &~(hn)^d \frac{1}{(hn)^{dk}}\left(\frac{|A|}{h}\right)^{hk} \frac{1}{k^k} = \frac{1}{k^k h^{d(k-1)+hk}}\frac{1}{n^{d(k-1)}}|A|^{kh}.
\end{align}

Let $M_2$ be the number of $(kh)$-tuples $(x_{1,1}, \ldots, x_{1,h}, \ldots, x_{k,1}, \ldots, x_{k,h})\in A^{kh}$ such that:
(1)~$\sum_{i=1}^{h}x_{1,i}=\sum_{i=1}^{h}x_{2,i}=\cdots =\sum_{i=1}^{h}x_{k,i}$;
(2)~$x_{\ell,i}$ ($\ell \in [k]$, $i\in [h]$) are not pairwise distinct.
Since the number of ways to partition a $(kh)$-element set into $k$ groups each containing $h$ elements is $\frac{{kh \choose h}{kh-h \choose h}\cdots {h\choose h}}{k!}$, we have
\begin{equation}\label{eq:fdkh}
f_{d,k,h}(A)\geq (M_1-M_2)/\frac{{kh \choose h}{kh-h \choose h}\cdots {h\choose h}}{k!}.
\end{equation}

\begin{claim}\label{cl:fdkh}
$M_2\leq {kh\choose 2}|A|^{k(h-1)}$.
\end{claim}

\begin{proof}
Note that if $x_{\ell,i}$ ($\ell \in [k]$, $i\in [h]$) are not pairwise distinct, then there exists a point $y\in A$ which appears at least twice among these $kh$ points.
There are $|A|$ choices of a point $y$ in $A$, and ${kh\choose 2}$ choices of two points in
$x_{1,1}, \ldots, x_{1,h}, \ldots, x_{k,1}, \ldots, x_{k,h}$ to be $y$.
We now fix a choice of $y$ and a choice of the two points arbitrarily.
Without loss of generality, if the two points have the same first subscript, we may assume that they are $x_{1,1}$ and $x_{1,2}$;
if the two points have distinct first subscripts, we may assume that they are $x_{1,1}$ and $x_{2,1}$.
Then there are at most $|A|^{k(h-1)-1}$ choices of
$x_{1,i}$ ($1\leq i\leq h$) and $x_{\ell,i}$ ($2\leq \ell\leq k$, $1\leq i\leq h-1$),
and after choosing these points, there is at most one choice of $x_{\ell,h}$ satisfying $\sum_{i=1}^{h}x_{1,i}=\sum_{i=1}^{h}x_{\ell,i}$ for each $2\leq \ell\leq k$.
Hence, we have $M_2\leq |A|{kh\choose 2}|A|^{k(h-1)-1}={kh\choose 2}|A|^{k(h-1)}$.
\end{proof}

By Inequality~(\ref{eq:fdkhM1}) and Claim~\ref{cl:fdkh}, we have
$$M_1-M_2\geq \frac{1}{k^k h^{d(k-1)+hk}}\frac{1}{n^{d(k-1)}}|A|^{kh}-{kh\choose 2}|A|^{k(h-1)}.$$
Moreover, since $|A|\geq n^{\frac{k-1}{k}d}\log n$ and $n$ is sufficiently large, we have $\frac{|A|^{k(h-1)}}{\frac{1}{n^{d(k-1)}}|A|^{kh}}=\frac{n^{d(k-1)}}{|A|^k}\leq \frac{n^{d(k-1)}}{\left(n^{\frac{k-1}{k}d}\log n\right)^k}=\frac{1}{(\log n)^k}=o(1).$
Hence, we have $$M_1-M_2\geq (1-o(1))\frac{1}{k^k h^{d(k-1)+hk}}\frac{|A|^{kh}}{n^{d(k-1)}}.$$
By setting $c_{\ref{le:fdkh}}\colonequals \frac{(1-o(1))k!}{k^k h^{d(k-1)+hk} {kh \choose h}{kh-h \choose h}\cdots {h\choose h}}$ and recalling Inequality~(\ref{eq:fdkh}), we further have
$$f_{d,k,h}(A)\geq (1-o(1))\frac{1}{k^k h^{d(k-1)+hk}}\frac{|A|^{kh}}{n^{d(k-1)}}/\frac{{kh \choose h}{kh-h \choose h}\cdots {h\choose h}}{k!} =c_{\ref{le:fdkh}}\frac{|A|^{kh}}{n^{d(k-1)}}.$$
The proof of Lemma~\ref{le:fdkh} is complete.
\end{proof}

For two disjoint subsets $A_1, A_2\subseteq [n]^d$ and an integer $j\in[h-1]$,
define $f_{d,k,h,j}(A_1,A_2)$ as the number of distinct subsets $\{x_{\ell,i}\colon\, 1\leq \ell\leq k, 1\leq i\leq h\}$
with $\{x_{\ell,i}\colon\, 1\leq \ell\leq k, 1\leq i\leq j\}\subseteq A_1$, $\{x_{\ell,i}\colon\, 1\leq \ell\leq k, j+1\leq i\leq h\}\subseteq A_2$
and $\sum_{i=1}^{h}x_{1,i}=\sum_{i=1}^{h}x_{2,i}=\cdots =\sum_{i=1}^{h}x_{k,i}$.
We now prove a lower bound on $f_{d,k,h,j}(A_1, A_2)$ by an analogous argument as the proof of Lemma~\ref{le:fdkh}.

\begin{lemma}\label{le:fdkhj}
Let $d, k, h, j, n$ be integers with $d\geq 1$, $h\geq 2$, $1\leq j\leq h-1$ and $n$ sufficiently large.
For any two disjoint subsets $A_1, A_2\subseteq [n]^d$ with $|A_1|, |A_2|\geq n^{\frac{k-1}{k}d}\log n$, there exists a constant $c_{\ref{le:fdkhj}}=c_{\ref{le:fdkhj}}(d,k,h,j)>0$ such that
$$f_{d,k,h,j}(A_1, A_2) \geq c_{\ref{le:fdkhj}}\frac{1}{n^{d(k-1)}}|A_1|^{kj}|A_2|^{k(h-j)}.$$
\end{lemma}

\begin{proof}
For any point $s\in [hn]^d$, let $m_{A_1, A_2}(s)$ be the number of subsets $\{x_1, \ldots, x_{j}\}\cup\{x_{j+1}, \ldots, x_{h}\}$ with $x_1, \ldots, x_{j}\in A_1$, $x_{j+1}, \ldots, x_{h}\in A_2$ and $x_1+\cdots +x_h=s$.
Then $\sum_{s\in [hn]^d}m_{A_1, A_2}(s)={|A_1|\choose j}{|A_2|\choose h-j}\geq \frac{1}{j^j(h-j)^{h-j}}|A_1|^j|A_2|^{h-j}$.
Let $M_1$ be the number of sets $\big\{\{x_{1,1}, \ldots, x_{1,h}\}, \ldots, \{x_{k,1}, \ldots, x_{k,h}\}\big\}$ consisting of $k$ subsets
with $\{x_{\ell,i}\colon\, 1\leq \ell\leq k, 1\leq i\leq j\}\subseteq A_1$, $\{x_{\ell,i}\colon\, 1\leq \ell\leq k, j+1\leq i\leq h\}\subseteq A_2$
and $\sum_{i=1}^{h}x_{1,i}=\sum_{i=1}^{h}x_{2,i}=\cdots =\sum_{i=1}^{h}x_{k,i}$, where $x_{\ell,i}$ ($\ell \in [k]$, $i\in [h]$) are not necessarily pairwise distinct.
By Jensen's inequality, we have
\begin{align}\label{eq:fdkhj}
  M_1= &~\sum_{s\in [hn]^d} {m_{A_1, A_2}(s) \choose k}\geq (hn)^d{\frac{1}{(hn)^d}\sum_{s\in [hn]^d}m_{A_1, A_2}(s) \choose k}\nonumber\\
   \geq &~(hn)^d{\frac{1}{(hn)^d}\frac{1}{j^j(h-j)^{h-j}}|A_1|^j|A_2|^{h-j} \choose k}\geq \frac{1}{k^kh^{d(k-1)}j^{kj}(h-j)^{k(h-j)}}\frac{|A_1|^{kj}|A_2|^{k(h-j)}}{n^{d(k-1)}}.
\end{align}

Let $M_2$ be the number of $(kh)$-tuples $(x_{1,1}, \ldots, x_{1,h}, \ldots, x_{k,1}, \ldots, x_{k,h})$
with $\{x_{\ell,i}\colon\, 1\leq \ell\leq k, 1\leq i\leq j\}\subseteq A_1$ and $\{x_{\ell,i}\colon\, 1\leq \ell\leq k, j+1\leq i\leq h\}\subseteq A_2$
such that:
(1)~$\sum_{i=1}^{h}x_{1,i}=\sum_{i=1}^{h}x_{2,i}=\cdots =\sum_{i=1}^{h}x_{k,i}$;
(2)~$x_{\ell,i}$ ($\ell \in [k]$, $i\in [h]$) are not pairwise distinct.

\begin{claim}\label{cl:fdkhj}
$M_2\leq {kh \choose 2}\cdot \max\left\{|A_1|^{kj-1}|A_2|^{k(h-j)-(k-1)}, |A_1|^{kj-(k-1)}|A_2|^{k(h-j)-1}\right\}$.
\end{claim}

\begin{proof}
If $x_{\ell,i}$ ($\ell \in [k]$, $i\in [h]$) are not pairwise distinct, then there exists a point $y\in A_1\cup A_2$ which appears at least twice among these $kh$ points.
For the case $y\in A_1$, there are $|A_1|$ choices of a point $y$ in $A_1$,
and ${kj\choose 2}$ choices of two points in $x_{1,1}, \ldots, x_{1,j}, \ldots, x_{k,1}, \ldots, x_{k,j}$ to be $y$.
Fix a choice of $y$ and a choice of the two points.
Then there are at most $|A_1|^{kj-2}|A_2|^{k(h-j)-(k-1)}$ choices of
$x_{1,i}$ ($1\leq i\leq h$) and $x_{\ell,i}$ ($2\leq \ell\leq k$, $1\leq i\leq h-1$),
and after choosing these points, there is at most one choice of $x_{\ell,h}$ satisfying $\sum_{i=1}^{h}x_{1,i}=\sum_{i=1}^{h}x_{\ell,i}$ for each $2\leq \ell\leq k$.
Thus there is a contribution of at most $|A_1|{kj\choose 2}|A_1|^{kj-2}|A_2|^{k(h-j)-(k-1)}={kj\choose 2}|A_1|^{kj-1}|A_2|^{k(h-j)-(k-1)}$ to $M_2$ in this case.

For the case $y\in A_2$, there are $|A_2|$ choices of a point $y$ in $A_2$,
and ${k(h-j)\choose 2}$ choices of two points in $x_{1,j+1}, \ldots, x_{1,h}, \ldots, x_{k,j+1}, \ldots, x_{k,h}$ to be $y$.
Fix a choice of $y$ and a choice of the two points.
Then there are at most $|A_1|^{kj-(k-1)}|A_2|^{k(h-j)-2}$ choices of
$x_{1,i}$ ($1\leq i\leq h$) and $x_{\ell,i}$ ($2\leq \ell\leq k$, $2\leq i\leq h$),
and after choosing these points, there is at most one choice of $x_{\ell,1}$ satisfying $\sum_{i=1}^{h}x_{1,i}=\sum_{i=1}^{h}x_{\ell,i}$ for each $2\leq \ell\leq k$.
Thus there is a contribution of at most $|A_2|{k(h-j)\choose 2}|A_1|^{kj-(k-1)}|A_2|^{k(h-j)-2}={k(h-j)\choose 2}|A_1|^{kj-(k-1)}|A_2|^{k(h-j)-1}$ to $M_2$ in this case.

Therefore, we have
\begin{align*}
M_2\leq &~{kj\choose 2}|A_1|^{kj-1}|A_2|^{k(h-j)-(k-1)}+{k(h-j)\choose 2}|A_1|^{kj-(k-1)}|A_2|^{k(h-j)-1} \\
\leq &~{kh \choose 2}\cdot \max\left\{|A_1|^{kj-1}|A_2|^{k(h-j)-(k-1)}, |A_1|^{kj-(k-1)}|A_2|^{k(h-j)-1}\right\}.
\end{align*}
The proof of Claim~\ref{cl:fdkhj} is complete.
\end{proof}

Since $|A_1|, |A_2|\geq n^{\frac{k-1}{k}d}\log n$ and $n$ is sufficiently large, we have
$$\frac{|A_1|^{kj-1}|A_2|^{k(h-j)-(k-1)}}{\frac{|A_1|^{kj}|A_2|^{k(h-j)}}{n^{d(k-1)}}}=\frac{n^{d(k-1)}}{|A_1||A_2|^{k-1}}\leq \frac{n^{d(k-1)}}{\left(n^{\frac{k-1}{k}d}\log n\right)^{k}}=\frac{1}{(\log n)^k}=o(1)$$
and
$$\frac{|A_1|^{kj-(k-1)}|A_2|^{k(h-j)-1}}{\frac{|A_1|^{kj}|A_2|^{k(h-j)}}{n^{d(k-1)}}}=\frac{n^{d(k-1)}}{|A_1|^{k-1}|A_2|}\leq \frac{n^{d(k-1)}}{\left(n^{\frac{k-1}{k}d}\log n\right)^{k}}=\frac{1}{(\log n)^k}=o(1).$$
Combining with Inequality~(\ref{eq:fdkhj}) and Claim~\ref{cl:fdkhj}, we further have
$$M_1-M_2 \geq (1-o(1))\frac{1}{k^kh^{d(k-1)}j^{kj}(h-j)^{k(h-j)}}\frac{|A_1|^{kj}|A_2|^{k(h-j)}}{n^{d(k-1)}}.$$
Since the number of ways to partition a $(kh)$-element set into $k$ groups each containing $h$ elements is $\frac{{kh \choose h}{kh-h \choose h}\cdots {h\choose h}}{k!}$, we further have
\begin{align*}
  f_{d,k,h,j}(A_1, A_2)\geq &~(M_1-M_2)/\frac{{kh \choose h}{kh-h \choose h}\cdots {h\choose h}}{k!}\\
   \geq &~(1-o(1))\frac{k!}{k^kh^{d(k-1)}j^{kj}(h-j)^{k(h-j)}{kh \choose h}{kh-h \choose h}\cdots {h\choose h}}\frac{|A_1|^{kj}|A_2|^{k(h-j)}}{n^{d(k-1)}}.
\end{align*}
This completes the proof of Lemma~\ref{le:fdkhj}.
\end{proof}

For any subset $A\subseteq [n]^d$ and point $v\in A$, define $f_{A}(v)$ as the number of distinct subsets $\{x_{1,i}\colon\, 2\leq i\leq h\}\cup \{x_{\ell,i}\colon\, 2\leq \ell\leq k, 1\leq i\leq h\}$ of $A\setminus \{v\}$ such that $v+\sum_{i=2}^{h}x_{1,i}=\sum_{i=1}^{h}x_{2,i}=\cdots =\sum_{i=1}^{h}x_{k,i}$.
We first consider $f_{[n]^d}(v)$ for a corner point $v$ (i.e., its coordinates are either $1$ or $n$) in $[n]^d$.
Then, we provide a lower bound on $f_{A}(v)$ for any point $v$ in a subset $A\subseteq [n]^d$.

\begin{lemma}\label{le:fAv}
Let $d, k, h, n$ be integers with $d\geq 1$, $k\geq 2$, $h\geq 2$ and $n$ sufficiently large.
Then there exists a constant $c_{\ref{le:fAv}}=c_{\ref{le:fAv}}(d,k,h)>0$ such that
for any point $v\in [n]^d$ whose all coordinates are either $1$ or $n$, we have
$$f_{[n]^d}(v)\geq c_{\ref{le:fAv}}n^{dk(h-1)}.$$
\end{lemma}

\begin{proof}
Without loss of generality, we may assume that $v=(v(1), \ldots, v(d))$ and for certain $0\leq t\leq d$, we have $v(j)=1$ if $j\leq t$, and $v(j)=n$ if $j\geq t+1$.
Let $a\colonequals \frac{1}{2h(2k-1)}$.
We define $2k-1$ pairwise disjoint subsets $A_1, A_2, \ldots, A_k, B_2, \ldots, B_k$ of $[n]^d$ as follows:
\begin{align*}
  A_1\colonequals \big\{x=(x(1), \ldots, x(d))\in [n]^d\colon\, &~\mbox{$n-an< x(j)\leq n$ for $j\leq t$}, \\
  &~\mbox{$1\leq x(j)< an$ for $j\geq t+1$}\big\},
\end{align*}
\begin{align*}
  A_{\ell}\colonequals \big\{x=(x(1), \ldots, x(d))\in [n]^d\colon\, &~\mbox{$n-(2\ell-1)an< x(j)\leq n-(2\ell-2)an$ for $j\leq t$}, \\
  &~\mbox{$(2\ell-2)an\leq x(j)< (2\ell-1)an$ for $j\geq t+1$}\big\}
\end{align*}
and
\begin{align*}
  B_{\ell}\colonequals &~\big\{x=(x(1), \ldots, x(d))\in [n]^d\colon\, \\
  &~\mbox{~~~~$1+(h-1)(2\ell-3)an< x(j)< 1+(h-1)(2\ell-1)an$ for $j\leq t$}, \\
  &~\mbox{~~~~$n+h-1-(h-1)(2\ell-1)an< x(j)< n-(h-1)(2\ell-3)an$ for $j\geq t+1$}\big\}
\end{align*}
for $\ell=2, \ldots, k$.
Note that the sizes of $|A_1|, |A_2|, \ldots, |A_k|$ are $(1-o(1))(an)^d$.
We now check that these $2k-1$ subsets are pairwise disjoint.
By the definition, the subsets $A_1, A_2, \ldots, A_k$ are pairwise disjoint, and the subsets $B_2, \ldots, B_k$ are pairwise disjoint.
Moreover, since $a= \frac{1}{2h(2k-1)}$, we have
$n-(2k-1)an>1+(h-1)(2k-1)an$ for the $j$th coordinate with $j\leq t$,
and $(2k-1)an<n+h-1-(h-1)(2k-1)an$ for the $j$th coordinate with $j\geq t+1$.
Thus $A_{\ell_1}$ and $B_{\ell_2}$ are disjoint for any $1\leq \ell_1\leq k$ and $2\leq \ell_2\leq k$.

We shall choose $\{x_{\ell, 2}, \ldots, x_{\ell, h}\}$ from $A_{\ell}$ for each $1\leq \ell \leq k$,
and show that the point $x_{\ell, 1}$ must be contained in $B_{\ell}$ for each $2\leq \ell \leq k$.
For each $1\leq \ell \leq k$, we fix a set $\{x_{\ell, 2}, \ldots, x_{\ell, h}\}$ of points of $A_{\ell}$ arbitrarily.
Let $s_1\colonequals v+x_{1,2}+\cdots +x_{1,h}$, and for each $2\leq \ell \leq k$, let $s_{\ell}\colonequals x_{\ell, 2}+ \cdots+ x_{\ell, h}$ and $x_{\ell, 1}\colonequals s_1-s_{\ell}$.
Then for $j\leq t$, the $j$th coordinates of $s_1$ and $s_{\ell}$ satisfy
$$1+(h-1)(n-an)< s_1(j) \leq 1+(h-1)n$$
and
$$(h-1)(n-(2\ell-1)an)< s_{\ell}(j) \leq (h-1)(n-(2\ell-2)an)$$
for $2\leq \ell \leq k$,
so
$$x_{\ell, 1}(j)>1+(h-1)(n-an)-(h-1)(n-(2\ell-2)an)= 1+(h-1)(2\ell-3)an$$
and
$$x_{\ell, 1}(j)<1+(h-1)n-(h-1)(n-(2\ell-1)an)= 1+(h-1)(2\ell-1)an.$$
For $j\geq t+1$, we have
$$n+h-1\leq s_1(j) < n+(h-1)an$$
and
$$(h-1)(2\ell-2)an \leq s_{\ell}(j)< (h-1)(2\ell-1)an$$
for $2\leq \ell \leq k$,
so
$$x_{\ell, 1}(j)>n+h-1-(h-1)(2\ell-1)an$$
and
$$x_{\ell, 1}(j)<n+(h-1)an-(h-1)(2\ell-2)an= n-(h-1)(2\ell-3)an.$$
Hence, we have $x_{\ell, 1}\in B_{\ell}$ for each $2\leq \ell \leq k$.

Therefore, $$f_{[n]^d}(v)\geq \prod_{\ell=1}^{k}{|A_{\ell}|\choose h-1}\geq (1-o(1)){(an)^d \choose h-1}^k\geq (1-o(1))\frac{a^{dk(h-1)}}{(h-1)^{k(h-1)}}n^{dk(h-1)}.$$
This completes the proof of Lemma~\ref{le:fAv}.
\end{proof}

\begin{corollary}\label{cor:fAv}
Let $d, k, h, n$ be integers with $d\geq 1$, $k\geq 2$, $h\geq 2$ and $n$ sufficiently large, and let $A\subseteq [n]^d$ be subset with $|A|=n^d-o(n^d)$.
Then there exists a constant $c_{\ref{cor:fAv}}=c_{\ref{cor:fAv}}(d,k,h)>0$ such that
for any point $v\in A$, we have
$$f_{A}(v)\geq c_{\ref{cor:fAv}}n^{dk(h-1)}.$$
\end{corollary}

\begin{proof}
Without loss of generality, we may assume that $v=(v(1), \ldots, v(d))$ and for certain $0\leq t\leq d$, we have $v(j)\leq \left\lfloor\frac{n}{2}\right\rfloor$ if $j\leq t$, and $v(j)\geq \left\lfloor\frac{n}{2}\right\rfloor+1$ if $j\geq t+1$.
Let
\begin{align*}
  F\colonequals \big\{x=(x(1), \ldots, x(d))\in [n]^d\colon\, &~\mbox{$v(j)\leq x(j)\leq v(j)+\left\lfloor\frac{n}{2}\right\rfloor-1$ for $j\leq t$}, \\
  &~\mbox{$v(j)-\left\lfloor\frac{n}{2}\right\rfloor+1\leq x(j)\leq v(j)$ for $j\geq t+1$}\big\}.
\end{align*}
Then $F$ is isomorphic to the grid $\left[\left\lfloor\frac{n}{2}\right\rfloor\right]^{d}$
in which $v$ corresponds to the corner point $v'\in \left[\left\lfloor\frac{n}{2}\right\rfloor\right]^{d}$ with $v'(j)=1$ for $j\leq t$ and $v'(j)=\left\lfloor\frac{n}{2}\right\rfloor$ for $j\geq t+1$.
By Lemma~\ref{le:fAv}, there exists a constant $c_{\ref{le:fAv}}>0$ such that $f_F(v)=f_{\left[\left\lfloor\frac{n}{2}\right\rfloor\right]^{d}}(v')\geq c_{\ref{le:fAv}}\left\lfloor\frac{n}{2}\right\rfloor^{dk(h-1)}\geq \frac{c_{\ref{le:fAv}}}{2^{dkh}}n^{dk(h-1)}$.

We next consider solutions to $v+\sum_{i=2}^{h}x_{1,i}=\sum_{i=1}^{h}x_{2,i}=\cdots =\sum_{i=1}^{h}x_{k,i}$ with a point in $[n]^d\setminus A$.
Since $|A|=n^d-o(n^d)$, we have $|[n]^d\setminus A|=o(n^d)$.
Fixing a point in $[n]^d\setminus A$, we set it as one point in $\{x_{1,i}\colon\, 2\leq i\leq h\}\cup \{x_{\ell,i}\colon\, 2\leq \ell\leq k, 1\leq i\leq h\}$, say $x_{1,2}$.
Then there are at most $|A|^{h-2+(h-1)(k-1)}$ choices of $\{x_{1,i}\colon\, 3\leq i\leq h\}\cup \{x_{\ell,i}\colon\, 2\leq \ell\leq k, 1\leq i\leq h-1\}$,
and after choosing these points, there is at most one choice of $x_{\ell,h}$ satisfying $v+\sum_{i=2}^{h}x_{1,i}=\sum_{i=1}^{h}x_{\ell,i}$ for each $2\leq \ell\leq k$.
Hence, there are at most $o(n^d)|A|^{h-2+(h-1)(k-1)}$ solutions with a point in $[n]^d\setminus A$.
Therefore, we have $$f_{A}(v)\geq f_F(v)-o(n^d)|A|^{h-2+(h-1)(k-1)}\geq \frac{c_{\ref{le:fAv}}}{2^{dkh}}n^{dk(h-1)}-o(n^{dk(h-1)})= (1-o(1))\frac{c_{\ref{le:fAv}}}{2^{dkh}}n^{dk(h-1)}.$$
The proof of Corollary~\ref{cor:fAv} is complete.
\end{proof}

\subsection{Container theorems}
\label{subsec:container}

In this section, we first introduce the Hypergraph Container Method which was developed by Balogh-Morris-Samotij~\cite{BaMS} and Saxton-Thomason~\cite{SaTh} independently in 2015.
Then we state and prove a container theorem with respect to forbidden rainbow solutions to the $B_{k,h}$-equation.
The hypergraph container method is a powerful tool in combinatorics and has applications in various fields including
additive combinatorics, extremal graph theory, Ramsey theory, number theory and discrete geometry.
We first introduce some additional notation.

Let $\mathcal{H}$ be a $k$-uniform hypergraph with vertex set $V(H)$ and edge set $E(H)$.
Let $d(\mathcal{H})$ be the {\it average degree} of $\mathcal{H}$.
For any subset $U\subseteq V(\mathcal{H})$, let $\mathcal{H}[U]$ be the subhypergraph of $\mathcal{H}$ induced by $U$.
For any subset $U\subseteq V(\mathcal{H})$ with $2\leq |U|\leq k$, let $d_{\mathcal{H}}(U)\colonequals |\{e\in E(\mathcal{H})\colon\, U\subseteq e\}|$ be the {\it co-degree} of $U$.
For $2\leq j\leq k$, the {\it $j$th maximum co-degree} of $\mathcal{H}$ is $\Delta_j(\mathcal{H})\colonequals \max\{d_{\mathcal{H}}(U)\colon\, U\subseteq V(\mathcal{H}), |U|=j\}$.
For $0<\tau <1$, the {\it co-degree function} is defined as
$$\Delta(\mathcal{H},\tau)\colonequals 2^{\binom{k}{2}-1}\sum^{k}_{j=2}2^{-\binom{j-1}{2}} \frac{\Delta_j(\mathcal{H})}{d(\mathcal{H})\tau^{j-1}}.$$

We will use the following form of the hypergraph container theorem~\cite[Corollary~3.6]{SaTh} (see also \cite[Theorem~3.1]{BaSo}).

\begin{theorem}\label{th:HCT} {\normalfont (Hypergraph container theorem~\cite{SaTh})}
Let $\mathcal{H}$ be a $k$-uniform hypergraph on vertex set $[N]$.
Let $0< \varepsilon, \tau< 1/2$. Suppose that $\tau  < 1/(200\cdot k \cdot k!^2)$ and $\Delta(\mathcal{H}, \tau) \leq  \varepsilon /(12k!)$.
Then there exists a positive constant $c_{\ref{th:HCT}} = c_{\ref{th:HCT}}(k) \leq 1000\cdot k \cdot k!^3$ and a collection $\mathcal{C}$ of vertex subsets such that
\begin{itemize}
  \item[{\rm (i)}] every independent set in $\mathcal{H}$ is a subset of some $C\in \mathcal{C}$;

  \item[{\rm (ii)}] for every $C \in \mathcal{C}$, $|E(\mathcal{H} [C])| \leq  \varepsilon \cdot |E(\mathcal{H})|$;

  \item[{\rm (iii)}] $\log |\mathcal{C}| \leq c_{\ref{th:HCT}}N\tau \cdot \log(1/\varepsilon) \cdot \log(1/\tau)$.
\end{itemize}
\end{theorem}

In order to prove our main result, we shall prove a container theorem with respect to forbidden rainbow solutions to the $B_{k,h}$-equation.
Before stating and proving our container theorem, we first introduce the concept of {\it templates}, which was first proposed by Falgas-Ravry, O'Connell and Uzzell~\cite{FaOU}.

\begin{definition}\label{de:template} {\normalfont (Template, palette, subtemplate, rainbow solution)}
{\rm Let $A\subseteq [n]^d$ and $r$ be a positive integer.
\begin{itemize}
\item[(1)] An {\it $r$-template} of $A$ is a function $P\colon\, A \to  2^{[r]}$, associating to each point $x$ in $A$ a list of colors $P(x) \subseteq [r]$. We refer to this set $P(x)$ as the {\it palette} available at $x$.
\item[(2)] Let $P_1, P_2$ be two $r$-templates of $A$. We say that $P_1$ is a {\it subtemplate} of $P_2$ (written as $P_1 \subseteq P_2$) if $P_1(x) \subseteq P_2(x)$ for every point $x\in A$.
\item[(3)] For an $r$-template $P$ of $A$, we say that $P$ is a {\it rainbow solution} to the $B_{k,h}$-equation if there exist $kh$ points $x_{\ell,i}$ ($\ell \in [k]$, $i\in [h]$) in $A$ with $\sum_{i=1}^{h}x_{1,i}=\sum_{i=1}^{h}x_{2,i}=\cdots =\sum_{i=1}^{h}x_{k,i}$ such that
    $|P(x_{\ell,i})|=1$ for every $\ell\in [k]$ and $i\in [h]$,
    $P(x)=\emptyset$ for every $x\in A\setminus \{x_{\ell,i}\colon\, 1\leq \ell\leq k, 1\leq i\leq h\}$, and
    the palettes $P(x_{\ell,i})$ ($\ell \in [k]$, $i\in [h]$) are pairwise distinct.
\item[(4)] For an $r$-template $P$, we say that $P$ contains no rainbow solutions to the $B_{k,h}$-equation if it contains no subtemplate that is a rainbow solution to the $B_{k,h}$-equation.
\end{itemize}}
\end{definition}

For an $r$-template $P$, we define $R(P)$ as the number of subtemplates of $P$ that are rainbow solutions to the $B_{k,h}$-equation.
Note that for any $A'\subseteq A \subseteq [n]^d$, an $r$-coloring of $A'$ can be viewed as an $r$-template $P$ of $A$ with $|P(x)|=1$ for each $x\in A'$ and $P(x)=\emptyset$ for each $x\in A\setminus A'$.
For a collection $\mathcal{P}$ of $r$-templates of $A$, we use $g(\mathcal{P}, A)$ to denote the number of $r$-colorings of $A$ that are subtemplates of some $P\in \mathcal{P}$ and contain no rainbow solutions to the $B_{k,h}$-equation.
If $\mathcal{P}$ consists of a single template $P$, i.e., $\mathcal{P}=\{P\}$, then we simply write $g(\mathcal{P}, A)$ as $g(P, A)$.

Now we have all the ingredients to present our container theorem for templates without rainbow solutions to the $B_{k,h}$-equation.

\begin{theorem}\label{th:BkhCT}
For integers $d, k, h, r, n$ and subset $A\subseteq [n]^d$ with $d\geq 1$, $k\geq 2$, $h\geq 2$, $r\geq kh$, $n$ sufficiently large and $|A|\geq \frac{n^d}{\log n}$, there exists a collection $\mathcal{C}$ of $r$-templates of $A$ such that
\begin{itemize}
  \item[{\rm (i)}] every $r$-template of $A$ without rainbow solutions to the $B_{k,h}$-equation is a subtemplate of some $P\in \mathcal{C}$;

  \item[{\rm (ii)}] for every $P \in \mathcal{C}$, $R(P) \leq \frac{1}{(\log n)^{5kh}}|A|^{k(h-1)+1}$;

  \item[{\rm (iii)}] $|\mathcal{C}| \leq 2^{|A|n^{-\frac{k(h-1)}{kh-1}d}(\log n)^{10}}$.
\end{itemize}
\end{theorem}

\begin{proof} Let $\mathcal{H}$ be a $(kh)$-uniform hypergraph with vertex set $A \times [r]$,
whose edges are all subsets of $kh$ vertices of the form $\{(x_i, \alpha_i)\colon\, 1\leq i \leq kh\}$ such that
\begin{itemize}
\item $x_1, x_2, \ldots, x_{kh}$ are pairwise distinct points in $A$, and $\alpha_{1}, \alpha_{2}, \ldots, \alpha_{kh}$ are pairwise distinct colors in $[r]$;
\item there exists a way to partition the elements of $\{x_1, x_2, \ldots, x_{kh}\}$ into $k$ groups of $h$ elements each, and the sum of elements in each group is equal.
\end{itemize}
In other words, every edge in $\mathcal{H}$ corresponds to a rainbow solutions to the $B_{k,h}$-equation in $A$.
Note that every vertex subset of $\mathcal{H}$ corresponds to an $r$-template of $A$, and every independent set in $\mathcal{H}$ corresponds to an $r$-template of $A$ without rainbow solutions to the $B_{k,h}$-equation.
Hence, it suffices to show that for appropriate $\varepsilon$ and $\tau$, there exists a collection $\mathcal{C}$ of vertex subsets satisfying Theorem~\ref{th:HCT}~(i), (ii), (iii).
To achieve this, we set
$$\varepsilon\colonequals \frac{(r-kh)!}{r!}\frac{1}{(\log n)^{5kh}} \mbox{~~~and~~~} \tau\colonequals n^{-\frac{k(h-1)}{kh-1}d}(\log n)^{8},$$
and we shall check that $\Delta(\mathcal{H}, \tau) \leq  \varepsilon /(12\cdot (kh)!)$ in the following arguments.

By Lemma~\ref{le:fdkh} and since there are exactly $r(r - 1)\cdots (r - kh+1)=\frac{r!}{(r-kh)!}$ ways to rainbow color $kh$ pairwise distinct points with $r$ colors, we have
$$|E(\mathcal{H})|=\frac{r!}{(r-kh)!}f_{d,k,h}(A)\geq \frac{r!}{(r-kh)!}c_{\ref{le:fdkh}}\frac{|A|^{kh}}{n^{d(k-1)}}.$$
Let $c\colonequals \frac{kh}{r}\frac{r!}{(r-kh)!}c_{\ref{le:fdkh}}$.
Then
\begin{equation}\label{eq:dH}
d(\mathcal{H})=\frac{kh|E(\mathcal{H})|}{|V(\mathcal{H})|}\geq\frac{kh\frac{r!}{(r-kh)!}c_{\ref{le:fdkh}}\frac{|A|^{kh}}{n^{d(k-1)}}}{r|A|}=c\frac{|A|^{kh-1}}{n^{d(k-1)}}.
\end{equation}
Next, we consider upper bounds on $\Delta_j(\mathcal{H})$ for $2\leq j\leq kh$.

\begin{claim}\label{cl:deltaj}
$\Delta_j(\mathcal{H})\leq
\left\{
   \begin{aligned}
    &r^{kh-j}|A|^{kh-j-(k-1)}, & & \mbox{for $2\leq j \leq h-1$},\\
    &2r^{kh-j}|A|^{kh-j-(k-\lfloor\frac{j}{h}\rfloor)}, & & \mbox{for $h\leq j\leq kh-1$},\\
    &1, & & \mbox{for $j=kh$}.
   \end{aligned}
   \right.$
\end{claim}

\begin{proof}
Note that $\Delta_{kh}(\mathcal{H})=1$.
In the following, we consider the case $2\leq j\leq kh-1$.
Let $B$ be a subset of $j$ points of $A$.
We first estimate the number of subsets $B'\in {A\setminus B \choose kh-j}$ such that $B\cup B'$ forms a set $\{x_{\ell,i}\colon\, 1\leq \ell\leq k, 1\leq i\leq h\}$ with $\sum_{i=1}^{h}x_{1,i}=\sum_{i=1}^{h}x_{2,i}=\cdots=\sum_{i=1}^{h}x_{k,i}$.
Let $t$ be the number of indices $\ell$ with $\{x_{\ell,i}\colon\, 1\leq i\leq h\}\subseteq B$.
Then $0\leq t\leq \lfloor\frac{j}{h}\rfloor$.

If $t=0$, then there are at most $|A|^{kh-j-(k-1)}$ choices of a subset of $kh-j-(k-1)$ points which together with $B$ plays the role of $\{x_{1,i}\colon\, 1\leq i\leq h\}\cup \{x_{\ell,i}\colon\, 2\leq \ell\leq k, 1\leq i\leq h-1\}$.
After choosing these points, there is at most one choice of $x_{\ell,h}$ satisfying $\sum_{i=1}^{h}x_{1,i}=\sum_{i=1}^{h}x_{\ell,i}$ for each $2\leq \ell\leq k$.
Thus there are at most $|A|^{kh-j-(k-1)}$ choices of a desired subset $B'$ when $t=0$.

If $t\geq 1$, then we may assume that $1, \ldots, t$ are these indices without loss of generality, i.e., $\bigcup_{\ell\in [t]}\{x_{\ell,i}\colon\, 1\leq i\leq h\}\subseteq B$.
Then there are at most $|A|^{kh-j-(k-t)}$ choices of a subset of $kh-j-(k-t)$ points which together with $B$ plays the role of $\{x_{\ell,i}\colon\, 1\leq \ell\leq t, 1\leq i\leq h\}\cup \{x_{\ell,i}\colon\, t+1\leq \ell\leq k, 1\leq i\leq h-1\}$.
After choosing these points, there is at most one choice of $x_{\ell,h}$ satisfying $\sum_{i=1}^{h}x_{1,i}=\sum_{i=1}^{h}x_{\ell,i}$ for each $t+1\leq \ell\leq k$.
Thus there are at most $|A|^{kh-j-(k-t)}$ choices of a desired subset $B'$ when $t\geq 1$.

Note that there are at most $r^{kh-j}$ ways to rainbow color $kh-j$ points.
If $2\leq j \leq h-1$, then $t=0$, and thus $\Delta_j(\mathcal{H})\leq r^{kh-j}|A|^{kh-j-(k-1)}$.
For $h\leq j \leq kh-1$, we have $\Delta_j(\mathcal{H})\leq r^{kh-j}\left(|A|^{kh-j-(k-1)}+\sum_{1\leq t\leq \lfloor\frac{j}{h}\rfloor}|A|^{kh-j-(k-t)}\right)$.
Since $\frac{|A|^{kh-j-(k-t)}}{|A|^{kh-j-(k-\lfloor\frac{j}{h}\rfloor)}}=o(1)$ for any $t\leq \lfloor\frac{j}{h}\rfloor-1$, we can further deduce that $\Delta_j(\mathcal{H})\leq 2r^{kh-j}|A|^{kh-j-(k-\lfloor\frac{j}{h}\rfloor)}$.
\end{proof}

\begin{claim}\label{cl:delta}
$\Delta(\mathcal{H}, \tau) \leq  \varepsilon /(12\cdot (kh)!)$.
\end{claim}

\begin{proof}
Recall that $|A|\geq \frac{n^d}{\log n}$, $\varepsilon= \frac{(r-kh)!}{r!}\frac{1}{(\log n)^{5kh}}$ and $\tau=n^{-\frac{k(h-1)}{kh-1}d}(\log n)^{8}$.
By Inequality~(\ref{eq:dH}), we have
\begin{align*}
  \Delta(\mathcal{H},\tau)= &~2^{\binom{kh}{2}-1}\sum^{kh}_{j=2}2^{-\binom{j-1}{2}} \frac{\Delta_j(\mathcal{H})}{d(\mathcal{H})\tau^{j-1}}\\
   \leq &~\sum^{kh}_{j=2}2^{\binom{kh}{2}-1-\binom{j-1}{2}}\frac{\Delta_j(\mathcal{H})}{c\frac{|A|^{kh-1}}{n^{d(k-1)}}\left(n^{-\frac{k(h-1)}{kh-1}d}(\log n)^{8}\right)^{j-1}}.
\end{align*}
For each $2\leq j \leq kh$, let $I_j\colonequals \frac{\Delta_j(\mathcal{H})}{\frac{|A|^{kh-1}}{n^{d(k-1)}}\left(n^{-\frac{k(h-1)}{kh-1}d}(\log n)^{8}\right)^{j-1}}.$
For $2\leq j \leq h-1$, we can deduce from Claim~\ref{cl:deltaj} that
\begin{align*}
  I_j \leq &~\frac{r^{kh-j}|A|^{kh-j-(k-1)}}{\frac{|A|^{kh-1}}{n^{d(k-1)}}\left(n^{-\frac{k(h-1)}{kh-1}d}(\log n)^{8}\right)^{j-1}}\\
  = &~r^{kh-j} n^{d(k-1)+\frac{k(h-1)}{kh-1}d(j-1)} |A|^{-(j-1)-(k-1)} (\log n)^{-8(j-1)}\\
  \leq &~r^{kh-j} n^{d(k-1)+\frac{k(h-1)}{kh-1}d(j-1)} \left(\frac{n^d}{\log n}\right)^{-(j-1)-(k-1)} (\log n)^{-8(j-1)}\\
  = &~r^{kh-j} n^{d(k-1)+\frac{k(h-1)}{kh-1}d(j-1)-d(j-1)-d(k-1)} (\log n)^{j-1+k-1-8(j-1)}\\
  = &~r^{kh-j} n^{-\frac{k-1}{kh-1}d(j-1)} (\log n)^{k-1-7(j-1)} = o\left((\log n)^{-5kh}\right).
\end{align*}
For $h\leq j\leq kh-1$, we can deduce from Claim~\ref{cl:deltaj} that
\begin{align*}
  I_j \leq &~\frac{2r^{kh-j}|A|^{kh-j-(k-\lfloor\frac{j}{h}\rfloor)}}{\frac{|A|^{kh-1}}{n^{d(k-1)}}\left(n^{-\frac{k(h-1)}{kh-1}d}(\log n)^{8}\right)^{j-1}}\\
  = &~2r^{kh-j} n^{d(k-1)+\frac{k(h-1)}{kh-1}d(j-1)} |A|^{-(j-1)-(k-\lfloor\frac{j}{h}\rfloor)} (\log n)^{-8(j-1)}\\
  \leq &~2r^{kh-j} n^{d(k-1)+\frac{k(h-1)}{kh-1}d(j-1)} \left(\frac{n^d}{\log n}\right)^{-(j-1)-(k-\lfloor\frac{j}{h}\rfloor)} (\log n)^{-8(j-1)}\\
  = &~2r^{kh-j} n^{d(k-1)+\frac{k(h-1)}{kh-1}d(j-1)-d(j-1)-d(k-\lfloor\frac{j}{h}\rfloor)} (\log n)^{(k-\lfloor\frac{j}{h}\rfloor)-7(j-1)}.
\end{align*}
Since
\begin{align*}
 ~ &~d(k-1)+\frac{k(h-1)}{kh-1}d(j-1)-d(j-1)-d\left(k-\left\lfloor\frac{j}{h}\right\rfloor\right)\\
 \leq &~\left(k-1-\frac{k(h-1)}{kh-1}+1-k\right)d+\left(\frac{k(h-1)}{kh-1}-1+\frac{1}{h}\right)dj\\
 = &~-\frac{k(h-1)}{kh-1}d+\frac{h-1}{(kh-1)h}dj \leq -\frac{k(h-1)}{kh-1}d+\frac{h-1}{(kh-1)h}d(kh-1) = -\frac{h-1}{(kh-1)h}d < 0,
\end{align*}
we have $I_j=o\left((\log n)^{-5kh}\right)$ for $h\leq j\leq kh-1$.
For $j= kh$, we can deduce from Claim~\ref{cl:deltaj} that
\begin{align*}
  I_j \leq &~\frac{1}{\frac{|A|^{kh-1}}{n^{d(k-1)}}\left(n^{-\frac{k(h-1)}{kh-1}d}(\log n)^{8}\right)^{kh-1}}\\
  = &~n^{d(k-1)+\frac{k(h-1)}{kh-1}d(kh-1)} |A|^{1-kh} (\log n)^{-8(kh-1)}\\
  \leq &~n^{d(k-1)+\frac{k(h-1)}{kh-1}d(kh-1)} \left(\frac{n^d}{\log n}\right)^{1-kh} (\log n)^{-8(kh-1)}\\
  = &~n^{d(k-1)+\frac{k(h-1)}{kh-1}d(kh-1)+d(1-kh)} (\log n)^{kh-1-8(kh-1)}\\
  = &~(\log n)^{-7kh+7} = o\left((\log n)^{-5kh}\right).
\end{align*}
Therefore, we have $\Delta(\mathcal{H},\tau)=o\left((\log n)^{-5kh}\right)$, and thus $\Delta(\mathcal{H}, \tau) \leq  \varepsilon /(12\cdot (kh)!)$.
\end{proof}

By Claim~\ref{cl:delta}, Theorem~\ref{th:HCT} holds.
In particular, by Theorem~\ref{th:HCT}~(ii) and Lemma~\ref{le:fdkh}, we have
\begin{align*}
  R(P) \leq &~\varepsilon \cdot |E(\mathcal{H})| = \varepsilon \cdot \frac{r!}{(r-kh)!}f_{d,k,h}(A)\\
  \leq &~\frac{(r-kh)!}{r!}\frac{1}{(\log n)^{5kh}} \cdot \frac{r!}{(r-kh)!}|A|^{k(h-1)+1} = \frac{1}{(\log n)^{5kh}}|A|^{k(h-1)+1}.
\end{align*}
By Theorem~\ref{th:HCT}~(iii), we have
$$|\mathcal{C}| \leq 2^{c_{\ref{th:HCT}}r|A|\tau \cdot \log(1/\varepsilon) \cdot \log(1/\tau)}\leq 2^{|A|n^{-\frac{k(h-1)}{kh-1}d}(\log n)^{10}}.$$
This completes the proof of Theorem~\ref{th:BkhCT}.
\end{proof}

\section{Stability analysis}
\label{sec:sta}

Throughout this section, we assume that the integers $d, h, r, n$ and the subset $A\subseteq [n]^d$ satisfy the assumption of Theorem~\ref{th:BkhCT}, and $\mathcal{C}$ is the collection of $r$-templates of $A$ given by Theorem~\ref{th:BkhCT}.
In other words, we have that $d\geq 1$, $k\geq 2$, $h\geq 2$, $r\geq kh$, $n$ sufficiently large, $|A|\geq \frac{n^d}{\log n}$, and Theorem~\ref{th:BkhCT}~(i), (ii), (iii) hold.
In this section, we show that $\mathcal{C}$ exhibits stability in the following sense: we can partition $\mathcal{C}$ into two parts $\mathcal{C}_{good}$ and $\mathcal{C}_{bad}$ such that
\begin{itemize}
\item[(1)] for any $r$-template $P\in \mathcal{C}_{good}$, there exists a set $C_P\in {[r]\choose kh-1}$ such that almost all points in $A$ have $C_P$ as their palettes (see Lemma~\ref{le:sta-CP});
\item[(2)] $g(\mathcal{C}_{bad}, A)$ is much smaller than the lower bound (\ref{eq:lower}) (see Lemma~\ref{le:sta-Cbad}).
\end{itemize}

For any $r$-template $P\in \mathcal{C}$ and integer $i\in [r]$, let
$$X_i(P)\colonequals \{x\in A\colon\, |P(x)|=i\},$$
$$X_{\leq kh-2}(P)\colonequals \{x\in A\colon\, |P(x)|\leq kh-2\}$$
and
$$X_{\geq kh}(P)\colonequals \{x\in A\colon\, |P(x)|\geq kh\}.$$

\begin{lemma}\label{le:sta-Xkh}
For any $r$-template $P\in \mathcal{C}$, we have $|X_{\geq kh}|\leq \frac{|A|}{(\log n)^4}.$
\end{lemma}

\begin{proof}
For a contradiction, suppose that $|X_{\geq kh}|> \frac{|A|}{(\log n)^4}.$
By Lemma~\ref{le:fdkh} and recalling that $|A|\geq \frac{n^d}{\log n}$, we have
\begin{align*}
f_{d,k,h}(X_{\geq kh})\geq &~c_{\ref{le:fdkh}}\frac{\left|X_{\geq kh}\right|^{kh}}{n^{d(k-1)}} > c_{\ref{le:fdkh}}\frac{\left(\frac{|A|}{(\log n)^4}\right)^{kh}}{n^{d(k-1)}} = c_{\ref{le:fdkh}}\frac{|A|^{kh}}{(\log n)^{4kh}n^{d(k-1)}}\\
\geq &~c_{\ref{le:fdkh}}\frac{|A|^{k(h-1)+1}}{(\log n)^{4kh}n^{d(k-1)}}\left(\frac{n^d}{\log n}\right)^{k-1} =c_{\ref{le:fdkh}}\frac{|A|^{k(h-1)+1}}{(\log n)^{4kh+k-1}} > \frac{1}{(\log n)^{5kh}}|A|^{k(h-1)+1}.
\end{align*}
Note that for any solution to $\sum_{i=1}^{h}x_{1,i}=\sum_{i=1}^{h}x_{2,i}=\cdots =\sum_{i=1}^{h}x_{k,i}$ with $\{x_{\ell,i}\colon\, 1\leq \ell\leq k, 1\leq i\leq h\}\in {X_{\geq kh} \choose kh}$, we have $|P(x_{\ell,i})|\geq kh$ for all $\ell\in [k]$ and $i\in [h]$.
Hence, we can greedily find $kh$ pairwise distinct colors
$\alpha_{\ell,i}\in P(x_{\ell,i})$ ($\ell \in [k]$, $i\in [h]$).
This implies that $R(P) \geq f_{d,k,h}(X_{\geq kh}) > \frac{1}{(\log n)^{5kh}}|A|^{k(h-1)+1}$, contradicting Theorem~\ref{th:BkhCT}~(ii).
\end{proof}

Let $\mathcal{C}_{good}\colonequals \left\{P\in \mathcal{C}\colon\, |X_{\leq kh-2}(P)|\leq \frac{|A|}{(\log n)^3}\right\}$ and $\mathcal{C}_{bad}\colonequals \mathcal{C}\setminus \mathcal{C}_{good}$.
Combining with Lemma~\ref{le:sta-Xkh}, we can obtain the following fact.

\begin{fact}\label{fact:sta-Xkh-1}
For any $r$-template $P\in \mathcal{C}_{good}$, we have $|X_{kh-1}(P)|\geq |A|-\frac{2|A|}{(\log n)^3}.$
\end{fact}

In fact, we can further show that for any $r$-template $P\in \mathcal{C}_{good}$, there exists a set $C_P\in {[r]\choose kh-1}$ such that almost all points in $A$ have $C_P$ as their palettes.
For an $r$-template $P$ and a set $C\subseteq [r]$, let $A_{P,C}\colonequals \{x\in A\colon\, P(x)=C\}$.

\begin{lemma}\label{le:sta-CP}
For any $r$-template $P\in \mathcal{C}_{good}$, there exists a set $C_P\in {[r]\choose kh-1}$ with $|A_{P,C_P}|\geq |A|-\frac{|A|}{(\log n)^2}$.
\end{lemma}

\begin{proof}
Let $C_P\in {[r]\choose kh-1}$ with $|A_{P,C_P}|=\max \left\{|A_{P,C}|\colon\, C\in {[r]\choose kh-1}\right\}$.
We shall show that for any $C\in {[r]\choose kh-1}\setminus \{C_P\}$, we have $|A_{P,C}|<\frac{|A|}{r^{kh}(\log n)^2}$.
This together with Fact~\ref{fact:sta-Xkh-1} implies that $|A_{P,C_P}|> |A|-\frac{2|A|}{(\log n)^3}-{r\choose kh-1}\frac{|A|}{r^{kh}(\log n)^2}>|A|-\frac{|A|}{(\log n)^2}.$
Suppose for a contradiction that there exists a $C\in {[r]\choose kh-1}\setminus \{C_P\}$ with $|A_{P,C}|\geq \frac{|A|}{r^{kh}(\log n)^2}$.
By the choice of $C_P$, we also have $|A_{P,C_P}| \geq \frac{|A|}{r^{kh}(\log n)^2}$.

Applying Lemma~\ref{le:fdkhj} with $j=h-1$ and recalling that $|A|\geq \frac{n^d}{\log n}$, we have
\begin{align*}
~&~f_{d,k,h,h-1}(A_{P,C_P}, A_{P,C}) \\
\geq &~c_{\ref{le:fdkhj}}\frac{1}{n^{d(k-1)}}|A_{P,C_P}|^{k(h-1)}|A_{P,C}|^{k} \geq c_{\ref{le:fdkhj}}\frac{1}{n^{d(k-1)}} \left(\frac{|A|}{r^{kh}(\log n)^2}\right)^{k(h-1)} \left(\frac{|A|}{r^{kh}(\log n)^2}\right)^k\\
= &~c_{\ref{le:fdkhj}}\frac{1}{n^{d(k-1)}}\left(\frac{|A|}{r^{kh}(\log n)^2}\right)^{kh} = c_{\ref{le:fdkhj}}\frac{1}{n^{d(k-1)}}\frac{|A|^{kh}}{r^{k^2h^2}(\log n)^{2kh}} \\
\geq &~c_{\ref{le:fdkhj}}\frac{1}{n^{d(k-1)}}\frac{|A|^{k(h-1)+1}}{r^{k^2h^2}(\log n)^{2kh}}\left(\frac{n^d}{\log n}\right)^{k-1} = c_{\ref{le:fdkhj}}\frac{|A|^{k(h-1)+1}}{r^{k^2 h^2}(\log n)^{2kh+k-1}} > \frac{1}{(\log n)^{5kh}}|A|^{k(h-1)+1}.
\end{align*}
Note that for any solution to $\sum_{i=1}^{h}x_{1,i}=\sum_{i=1}^{h}x_{2,i}=\cdots =\sum_{i=1}^{h}x_{k,i}$ with
$\{x_{\ell,i}\colon\, 1\leq \ell \leq k, 1\leq i\leq h-1\}\subseteq A_{P,C_P}$ and $\{x_{1,h}, \ldots, x_{k,h}\}\subseteq A_{P,C}$,
we have $P(x_{\ell,i})=C_P$ and $P(x_{\ell,h})=C$ for every $\ell\in [k]$ and $i\in [h-1]$.
Since $C\neq C_P$, we may assume that $C_P=\{\alpha_1, \ldots, \alpha_{kh-1}\}$, $\alpha^{\ast}_1\in C\setminus C_P$, $\alpha^{\ast}_2, \ldots, \alpha^{\ast}_k\in C\setminus \{\alpha^{\ast}_1\}$ and $\alpha^{\ast}_2, \ldots, \alpha^{\ast}_k\notin \{\alpha_1, \ldots, \alpha_{kh-k}\}$.
Let $P'$ be a subtemplate of $P$ with
$P'(x_{\ell,h})=\{\alpha^{\ast}_{\ell}\}$ for $\ell\in [k]$,
$\left\{P'(x_{\ell,i})\colon\, 1\leq \ell \leq k, 1\leq i\leq h-1\right\}=\left\{\{\alpha_i\}\colon\, 1\leq i\leq kh-k\right\}$,
and $P'(x)=\emptyset$ for $x\in A\setminus \{x_{\ell,i}\colon\, 1\leq \ell\leq k, 1\leq i\leq h\}$.
Then $P'$ is a rainbow solution to the $B_{k,h}$-equation.
Thus $P$ has at least $f_{d,k,h,h-1}(A_{P,C_P}, A_{P,C})>\frac{1}{(\log n)^{5kh}}|A|^{k(h-1)+1}$ subtemplates that are rainbow solutions to the $B_{k,h}$-equation, contradicting Theorem~\ref{th:BkhCT}~(ii).
\end{proof}

For $\mathcal{C}_{bad}$, we now show that $g(\mathcal{C}_{bad}, A)$ is much smaller than the lower bound on $g_r(A)$ given by Inequality~(\ref{eq:lower}).

\begin{lemma}\label{le:sta-Cbad}
There exists a constant $c_{\ref{le:sta-Cbad}}=c_{\ref{le:sta-Cbad}}(k,h)>0$ such that $$g(\mathcal{C}_{bad}, A)\leq 2^{-c_{\ref{le:sta-Cbad}}\frac{|A|}{(\log n)^3}}(kh-1)^{|A|}.$$
\end{lemma}

\begin{proof}
Given a $P\in \mathcal{C}_{bad}$ and $i\in [r]$, let $a_i\colonequals |X_i(P)|$, and let $a'=|X_{\leq kh-2}(P)|$ and $a''=|X_{\geq kh}(P)|$.
Then $a'> \frac{|A|}{(\log n)^3}$ by the definition of $\mathcal{C}_{bad}$, and $a''\leq \frac{|A|}{(\log n)^4}$ by Lemma~\ref{le:sta-Xkh}.
Since $|A|=a'+ a_{kh-1} + a''$ and $a'=\sum_{i=1}^{kh-2}a_i$, we have
\begin{align*}
g(P, A)\leq &~\left(\prod_{i=1}^{kh-2}i^{a_i}\right)(kh-1)^{a_{kh-1}}r^{a''} \leq (kh-2)^{a'}(kh-1)^{|A|-a'-a''}r^{a''}\\
= &~\left(\frac{kh-2}{kh-1}\right)^{a'}(kh-1)^{|A|}\left(\frac{r}{kh-1}\right)^{a''} = 2^{-a'\log \frac{kh-1}{kh-2} + a''\log \frac{r}{kh-1}}\cdot (kh-1)^{|A|}.
\end{align*}
By Theorem~\ref{th:BkhCT}~(iii), we further have
\begin{align*}
\sum_{P\in \mathcal{C}_{bad}}g(P, A)\leq &~|\mathcal{C}_{bad}|2^{-a'\log \frac{kh-1}{kh-2} + a''\log \frac{r}{kh-1}}\cdot (kh-1)^{|A|}\\
\leq &~|\mathcal{C}|2^{-a'\log \frac{kh-1}{kh-2} + a''\log \frac{r}{kh-1}}\cdot (kh-1)^{|A|}\\
\leq &~\left(2^{|A|n^{-\frac{k(h-1)}{kh-1}d}(\log n)^{10}}\right)\left(2^{-a'\log \frac{kh-1}{kh-2} + a''\log \frac{r}{kh-1}}\cdot (kh-1)^{|A|}\right)\\
= &~2^{|A|n^{-\frac{k(h-1)}{kh-1}d}(\log n)^{10}-a'\log \frac{kh-1}{kh-2} + a''\log \frac{r}{kh-1}} \cdot (kh-1)^{|A|}.
\end{align*}
Since $\log \frac{kh-1}{kh-2}>0$, $\log \frac{r}{kh-1}\geq \log \frac{kh}{kh-1}>0$, $a'> \frac{|A|}{(\log n)^3}$, $a''\leq \frac{|A|}{(\log n)^4}$ and $-\frac{k(h-1)}{kh-1}<0$, we have
\begin{align*}
~ &~|A|n^{-\frac{k(h-1)}{kh-1}d}(\log n)^{10}-a'\log \frac{kh-1}{kh-2} + a''\log \frac{r}{kh-1}\\
< &~|A|n^{-\frac{k(h-1)}{kh-1}d}(\log n)^{10}-\frac{|A|}{(\log n)^3}\log \frac{kh-1}{kh-2} + \frac{|A|}{(\log n)^4}\log \frac{r}{kh-1}\\
\leq &~-\frac{1}{2}\frac{|A|}{(\log n)^3}\log \frac{kh-1}{kh-2}.
\end{align*}
Let $c_{\ref{le:sta-Cbad}}\colonequals \frac{1}{2}\log \frac{kh-1}{kh-2}$.
Then $g(\mathcal{C}_{bad}, A)\leq \sum_{P\in \mathcal{C}_{bad}}g(P, A)\leq 2^{-c_{\ref{le:sta-Cbad}}\frac{|A|}{(\log n)^3}}(kh-1)^{|A|}.$
\end{proof}

\section{Completing the proof}
\label{sec:proof}

In this section, we complete the proof of Theorem~\ref{th:Bkh}.
We divide the rest of proof into two parts:
\begin{itemize}
\item[(1)] Counting and typical structures: for $|A|\geq n^d-\frac{n^d}{\log n}$, we show that $g_r(A)= (1+o(1))\binom{r}{kh-1}(kh-1)^{|A|}$, and all but a $o(1)$ proportion of $r$-colorings of $A$ without rainbow solutions to the $B_{k,h}$-equation are $(kh-1)$-colorings.
\item[(2)] Extremal structure: among all subsets of $[n]^d$, we show that $[n]^d$ is the unique subset admitting the maximum number of $r$-colorings without rainbow solutions to the $B_{k,h}$-equation.
\end{itemize}

\subsection{Counting and typical structures}
\label{subsec:counting}

Let $d, k, h, r, n$ be integers with $d\geq 1$, $k\geq 2$, $h\geq 2$, $r\geq kh$ and $n$ sufficiently large,
and let $A\subseteq [n]^d$ be a subset with $|A|\geq n^d-\frac{n^d}{\log n}$.
By Theorem~\ref{th:BkhCT}~(i), we have
\begin{equation}\label{eq:upper1}
g_r(A)= g(\mathcal{C}, A)\leq g(\mathcal{C}_{good}, A)+g(\mathcal{C}_{bad}, A).
\end{equation}
Lemma~\ref{le:sta-Cbad} implies that $g(\mathcal{C}_{bad}, A)=o\left((kh-1)^{|A|}\right)$.
We next consider $g(\mathcal{C}_{good}, A)$.

Let $\mathscr{G}$ be the set of all $r$-colorings of $A$ without rainbow solutions to the $B_{k,h}$-equation such that each of these $r$-colorings is a subtemplate of some template in $\mathcal{C}_{good}$.
Let ${[r]\choose kh-1}=\left\{C_1, \ldots, C_{r\choose kh-1}\right\}.$
By Lemma~\ref{le:sta-CP}, there exist ${r\choose kh-1}$ subsets $\mathscr{G}_1, \ldots, \mathscr{G}_{r\choose kh-1}$ with $\mathscr{G}=\mathscr{G}_1\cup\cdots\cup\mathscr{G}_{r\choose kh-1}$ such that,
for each $1\leq i\leq {r\choose kh-1}$ and any $r$-coloring in $\mathscr{G}_i$, at least $|A|-\frac{|A|}{(\log n)^2}$ points is colored by a color in $C_i$.

In the following, we first consider the {\it deviation gain}, that is, the number $g_{dev}(A, C_i)$ of $r$-colorings in $\mathscr{G}_i$ in which at least one point receives a color in $[r]\setminus C_i$.
Note that
\begin{equation}\label{eq:upper2}
g(\mathcal{C}_{good}, A)=|\mathscr{G}|\leq \sum_{1\leq i\leq {r\choose kh-1}}\left|\mathscr{G}_i\right| = \sum_{1\leq i\leq {r\choose kh-1}} \left((kh-1)^{|A|}+g_{dev}(A, C_i)\right).
\end{equation}
We now show that the deviation gain is much smaller than the number of $(kh-1)$-colorings.

\begin{lemma}\label{le:dev}
For any $1\leq i\leq {r\choose kh-1}$, we have $g_{dev}(A, C_i)= o\left((kh-1)^{|A|}\right).$
\end{lemma}

\begin{proof}
Let $\mathscr{G}_{dev}(A, C_i)$ be the set of $r$-colorings in $\mathscr{G}_i$ in which at least one point receives a color in $[r]\setminus C_i$.
Then for each $r$-coloring $c\in \mathscr{G}_{dev}(A, C_i)$, there exists a {\it deviation set} $M_c\colonequals \{x\in A\colon\, c(x)\notin C_i\}$.
From the above argument, we know that $|M_c|\leq \frac{|A|}{(\log n)^2}$.

Fixing a set $M\subseteq A$ with $1\leq |M|\leq \frac{|A|}{(\log n)^2}$ arbitrarily, we now consider the number of $r$-colorings in $\mathscr{G}_i$ with $M$ as the deviation set.
Note that there are at most $r^{|M|}$ ways to color the points in $M$.
In the following, we consider the number of ways to color the points in $A\setminus M$.

Let $v$ be an arbitrarily fixed point in the deviation set $M$.
Then $c(v)\notin C_i$.
Let $\mathscr{F}$ be the collection of distinct subsets $\{x_{1,i}\colon\, 2\leq i\leq h\}\cup \{x_{\ell,i}\colon\, 2\leq \ell\leq k, 1\leq i\leq h\}$ of $A\setminus M$ such that $v+\sum_{i=2}^{h}x_{1,i}=\sum_{i=1}^{h}x_{2,i}=\cdots =\sum_{i=1}^{h}x_{k,i}$.
Since $|A|\geq n^d-\frac{n^d}{\log n}$, we have $|A\setminus M|\geq |A|-\frac{|A|}{(\log n)^2}=n^d-o(n^d).$
By Corollary~\ref{cor:fAv}, we have $|\mathscr{F}|= f_{(A\setminus M)\cup \{v\}}(v)\geq c_{\ref{cor:fAv}}n^{dk(h-1)}$.

\begin{claim}\label{cl:dev}
There exists a constant $c_{\ref{cl:dev}}=c_{\ref{cl:dev}}(d,k,h)>0$ such that $\mathscr{F}$ contains at least $c_{\ref{cl:dev}}n^{d}$ pairwise disjoint elements.
\end{claim}

\begin{proof}
Let $F_1=\left\{x_{1,i}^{(1)}\colon\, 2\leq i\leq h\right\}\cup \left\{x_{\ell,i}^{(1)}\colon\, 2\leq \ell\leq k, 1\leq i\leq h\right\}\in \mathscr{F}$.
We claim that there are at most $(kh-1)|A\setminus M|^{k(h-1)-1}$ elements in $\mathscr{F}$ containing a point of $F_1$.
Indeed, given any point of the $kh-1$ points of $F_1$, say $x_{1,2}^{(1)}$,
there are at most $|A\setminus M|^{k(h-1)-1}$ choices of $\{x_{1,i}\colon\, 3\leq i\leq h\}\cup \{x_{\ell,i}\colon\, 2\leq \ell\leq k, 1\leq i\leq h-1\}$,
and after choosing these points, there is at most one choice of $x_{\ell,h}$ satisfying $v+x_{1,2}^{(1)}+\sum_{i=3}^{h}x_{1,i}=\sum_{i=1}^{h}x_{\ell,i}$ for each $2\leq \ell\leq k$.
Hence, there exists a subset $\mathscr{F}_1\subseteq \mathscr{F}$ with $|\mathscr{F}_1|\geq |\mathscr{F}|-(kh-1)|A\setminus M|^{k(h-1)-1}\geq |\mathscr{F}|-(kh-1)n^{d(k(h-1)-1)}$ such that for any $F\in \mathscr{F}_1$ we have $F_1\cap F=\emptyset$.
Let $F_2=\left\{x_{1,i}^{(2)}\colon\, 2\leq i\leq h\right\}\cup \left\{x_{\ell,i}^{(2)}\colon\, 2\leq \ell\leq k, 1\leq i\leq h\right\}\in \mathscr{F}_1$.
By an analogous argument as above, there exists a subset $\mathscr{F}_2\subseteq \mathscr{F}_1$ with $|\mathscr{F}_2|\geq |\mathscr{F}_1|-(kh-1)n^{d(k(h-1)-1)}$ such that $F_2\cap F=\emptyset$ for any $F\in \mathscr{F}_2$.
Let $F_3=\left\{x_{1,i}^{(3)}\colon\, 2\leq i\leq h\right\}\cup \left\{x_{\ell,i}^{(3)}\colon\, 2\leq \ell\leq k, 1\leq i\leq h\right\}\in \mathscr{F}_2$.
Continuing this process, we obtain pairwise disjoint elements $F_1, F_2, \ldots, F_{t}$ of $\mathscr{F}$ for some $t$.
Since $|\mathscr{F}|\geq c_{\ref{cor:fAv}}n^{dk(h-1)}$, we have $t\geq \left\lfloor\frac{c_{\ref{cor:fAv}}n^{dk(h-1)}}{(kh-1)n^{d(k(h-1)-1)}}\right\rfloor = c_{\ref{cl:dev}}n^{d}$ for some constant $c_{\ref{cl:dev}}>0$.
\end{proof}

By Claim~\ref{cl:dev}, $\mathscr{F}$ contains a subset $\mathscr{F}'$ of pairwise disjoint elements with $|\mathscr{F}'|\geq c_{\ref{cl:dev}}n^{d}$.
Let $A'\colonequals \bigcup_{F\in \mathscr{F}'} F$ and $A''\colonequals A\setminus (A'\cup M)$.
Then $|A'|=(kh-1)|\mathscr{F}'|\geq c_{\ref{cl:dev}}(kh-1)n^{d}$.
Since $M$ is the deviation set, we have $c(x)\in C_i$ for any point $x\in A\setminus M$.
Recall that $c(v)\notin C_i$.
For any $F\in \mathscr{F}'$, in order to avoid a rainbow solutions to the $B_{k,h}$-equation, the colors of the $kh-1$ points in $F$ are not pairwise distinct.
Thus there are at most $(kh-1)^{kh-1}-(kh-1)!$ ways to color such $kh-1$ points.
Hence, there are at most $((kh-1)^{kh-1}-(kh-1)!)^{|\mathscr{F}'|}$ ways to color the points in $A'$.
For the points in $A''$, there are at most $(kh-1)^{|A''|}$ ways to color them.
Therefore, the number $T(M)$ of $r$-colorings in $\mathscr{G}_i$ with $M$ as the deviation set satisfies
\begin{align*}
T(M) \leq &~r^{|M|}\cdot ((kh-1)^{kh-1}-(kh-1)!)^{|\mathscr{F}'|}\cdot (kh-1)^{|A''|} \\
= &~r^{|M|}\cdot ((kh-1)^{kh-1}-(kh-1)!)^{\frac{|A'|}{kh-1}}\cdot (kh-1)^{|A''|}\\
= &~\left(\frac{r}{kh-1}\right)^{|M|}(kh-1)^{|M|} \cdot \left(\frac{(kh-1)^{kh-1}-(kh-1)!}{(kh-1)^{kh-1}}\right)^{\frac{|A'|}{kh-1}}(kh-1)^{|A'|}\cdot (kh-1)^{|A''|}\\
= &~\left(\frac{r}{kh-1}\right)^{|M|}\left(\frac{(kh-1)^{kh-1}-(kh-1)!}{(kh-1)^{kh-1}}\right)^{\frac{|A'|}{kh-1}} (kh-1)^{|A|}.
\end{align*}
Let $\alpha\colonequals \frac{r}{kh-1}$ and $\frac{1}{\beta}\colonequals \frac{(kh-1)^{kh-1}-(kh-1)!}{(kh-1)^{kh-1}}$.
Since $\alpha >1$, $\frac{1}{\beta}<1$, $|M|\leq \frac{|A|}{(\log n)^2}\leq \frac{n^d}{(\log n)^2}$ and $|A'|\geq c_{\ref{cl:dev}}(kh-1)n^{d}$, we have
\begin{align*}
~&~\left(\frac{r}{kh-1}\right)^{|M|}\left(\frac{(kh-1)^{kh-1}-(kh-1)!}{(kh-1)^{kh-1}}\right)^{\frac{|A'|}{kh-1}} \\
\leq &~\alpha^{\frac{n^d}{(\log n)^2}}\left(\frac{1}{\beta}\right)^{\frac{c_{\ref{cl:dev}}(kh-1)n^{d}}{kh-1}} = 2^{(\log \alpha)\frac{n^d}{(\log n)^2} -c_{\ref{cl:dev}}(\log \beta)n^{d}}.
\end{align*}
Thus $T(M)\leq 2^{(\log \alpha)\frac{n^d}{(\log n)^2} -c_{\ref{cl:dev}}(\log \beta)n^{d}}\cdot (kh-1)^{|A|}.$
Then
\begin{align*}
g_{dev}(A, C_i)\leq &~\sum_{M\subseteq A, 1\leq |M|\leq \frac{|A|}{(\log n)^2}}T(M) \leq \sum_{i=1}^{\left\lfloor\frac{|A|}{(\log n)^2}\right\rfloor} {|A| \choose i}2^{(\log \alpha)\frac{n^d}{(\log n)^2} -c_{\ref{cl:dev}}(\log \beta)n^{d}}\cdot (kh-1)^{|A|}\\
\leq &~\left\lfloor\frac{|A|}{(\log n)^2}\right\rfloor{|A| \choose \left\lfloor\frac{|A|}{(\log n)^2}\right\rfloor} 2^{(\log \alpha)\frac{n^d}{(\log n)^2} -c_{\ref{cl:dev}}(\log \beta)n^{d}}\cdot (kh-1)^{|A|}\\
\leq &~\frac{|A|}{(\log n)^2} \left(\frac{e|A|}{\frac{|A|}{(\log n)^2}}\right)^{\frac{|A|}{(\log n)^2}} 2^{(\log \alpha)\frac{n^d}{(\log n)^2} -c_{\ref{cl:dev}}(\log \beta)n^{d}}\cdot (kh-1)^{|A|}\\
\leq &~\frac{n^d}{(\log n)^2}\left(e(\log n)^2\right)^{\frac{n^d}{(\log n)^2}} 2^{(\log \alpha)\frac{n^d}{(\log n)^2} -c_{\ref{cl:dev}}(\log \beta)n^{d}}\cdot (kh-1)^{|A|}\\
\leq &~2^{d\log n} 2^{3(\log\log n)\frac{n^d}{(\log n)^2}} 2^{(\log \alpha)\frac{n^d}{(\log n)^2} -c_{\ref{cl:dev}}(\log \beta)n^{d}}\cdot (kh-1)^{|A|} \\
\leq &~2^{\frac{n^d}{\log n} -c_{\ref{cl:dev}}(\log \beta)n^{d}}\cdot (kh-1)^{|A|} = o\left((kh-1)^{|A|}\right),
\end{align*}
where the fourth inequality holds since ${N\choose t}\leq \left(\frac{eN}{t}\right)^t$ for any integers $N\geq t\geq 1$.
This completes the proof of Lemma~\ref{le:dev}.
\end{proof}

By Lemmas~\ref{le:sta-Cbad}, \ref{le:dev}, Inequalities~(\ref{eq:upper1}) and (\ref{eq:upper2}), we have
\begin{align*}
g_r(A)= &~g(\mathcal{C}, A)\leq g(\mathcal{C}_{good}, A)+g(\mathcal{C}_{bad}, A)\\
\leq &~\sum_{1\leq i\leq {r\choose kh-1}} \left((kh-1)^{|A|}+g_{dev}(A, C_i)\right)+g(\mathcal{C}_{bad}, A)\\
\leq &~\binom{r}{kh-1}\left((kh-1)^{|A|}+o\left((kh-1)^{|A|}\right)\right)+2^{-c'\frac{|A|}{(\log n)^3}}(kh-1)^{|A|}\\
= &~(1+o(1))\binom{r}{kh-1}(kh-1)^{|A|}.
\end{align*}
Moreover, from the lower bound~(\ref{eq:lower}), we know that the number of $(kh-1)$-colorings of $A$ is $(1-o(1))\binom{r}{kh-1}(kh-1)^{|A|}$.
Therefore, all but a $o(1)$ proportion of $r$-colorings of $A$ without rainbow solutions to the $B_{k,h}$-equation are $(kh-1)$-colorings.

\subsection{Extremal structure}
\label{subsec:extrema}

In order to show that among all subsets of $[n]^d$, $[n]^d$ is the unique subset admitting the maximum number of $r$-colorings without rainbow solutions to the $B_{k,h}$-equation,
we consider three cases based on the size of $A\subseteq [n]^d$.
Note that $g_r([n]^d)= (1-o(1))\binom{r}{kh-1}(kh-1)^{n^d}\geq (1-o(1))kh(kh-1)^{n^d}$ since $r\geq kh$.
\vspace{0.2cm}

{\bf Sparse sets: $|A|\leq \frac{n^d}{\log n}$}
\vspace{0.2cm}

In this case, we have $g_r(A)\leq r^{|A|}\leq r^{\frac{n^d}{\log n}}=(kh-1)^{(\log_{kh-1}r)\frac{n^d}{\log n}}< (kh-1)^{n^d}< g_r([n]^d).$
\vspace{0.2cm}

{\bf Medium sets: $\frac{n^d}{\log n}\leq |A|\leq n^d-\frac{n^d}{\log n}$}
\vspace{0.2cm}

In this case, since $|A|\geq \frac{n^d}{\log n}$, all the results obtained in Section~\ref{sec:sta} hold for $A$.
In particular, for any $r$-template $P\in \mathcal{C}_{good}$, we have $g(P,A)\leq (kh-1)^{|A|-\frac{|A|}{(\log n)^2}}\cdot r^{\frac{|A|}{(\log n)^2}}$ by Lemma~\ref{le:sta-CP}.
Combining with Theorem~\ref{th:BkhCT}~(iii), we have
\begin{align*}
g(\mathcal{C}_{good}, A)\leq &~|\mathcal{C}|(kh-1)^{|A|-\frac{|A|}{(\log n)^2}}\cdot r^{\frac{|A|}{(\log n)^2}} \leq 2^{|A|n^{-\frac{k(h-1)}{kh-1}d}(\log n)^{10}}\cdot (kh-1)^{|A|-\frac{|A|}{(\log n)^2}}\cdot r^{\frac{|A|}{(\log n)^2}}\\
= &~(kh-1)^{(\log_{kh-1}2)|A|n^{-\frac{k(h-1)}{kh-1}d}(\log n)^{10}}\cdot (kh-1)^{|A|-\frac{|A|}{(\log n)^2}}\cdot (kh-1)^{(\log_{kh-1}r)\frac{|A|}{(\log n)^2}}.
\end{align*}
Since $|A|\leq n^d-\frac{n^d}{\log n}$,
we further have
\begin{align*}
~ &~\log_{kh-1}\left(g(\mathcal{C}_{good}, A)\right)\\
\leq &~(\log_{kh-1}2)\left(n^d-\frac{n^d}{\log n}\right)n^{-\frac{k(h-1)}{kh-1}d}(\log n)^{10} + \left(1-\frac{1}{(\log n)^2}\right)\left(n^d-\frac{n^d}{\log n}\right) + (\log_{kh-1}r)\frac{n^d-\frac{n^d}{\log n}}{(\log n)^2}\\
< &~n^{\frac{k-1}{kh-1}d}(\log n)^{10} + n^d-\frac{n^d}{\log n} + (\log_{kh-1}r)\frac{n^d}{(\log n)^2} \leq n^d-\frac{n^d}{2\log n}.
\end{align*}
Hence, $g(\mathcal{C}_{good}, A)\leq (kh-1)^{n^d-\frac{n^d}{2\log n}}<\frac{1}{2}(kh-1)^{n^d}.$
By Lemma~\ref{le:sta-Cbad}, we also have $g(\mathcal{C}_{bad}, A)\leq 2^{-c'\frac{|A|}{(\log n)^3}}(kh-1)^{|A|}<\frac{1}{2}(kh-1)^{n^d}.$
Thus
$$g_r(A)=g(\mathcal{C}, A)\leq g(\mathcal{C}_{good}, A)+g(\mathcal{C}_{bad}, A)<\frac{1}{2}(kh-1)^{n^d}+\frac{1}{2}(kh-1)^{n^d}=(kh-1)^{n^d}< g_r([n]^d).$$

{\bf Dense sets: $|A|\geq n^d-\frac{n^d}{\log n}$}
\vspace{0.2cm}

If $n^d-\frac{n^d}{\log n} \leq |A|\leq n^d-1$, then the arguments in Section~\ref{subsec:counting} imply that
\begin{align*}
g_r(A)= &~(1+o(1))\binom{r}{kh-1}(kh-1)^{|A|}\leq (1+o(1))\binom{r}{kh-1}(kh-1)^{n^d-1}\\
= &~\frac{1}{kh-1}(1+o(1))\binom{r}{kh-1}(kh-1)^{n^d} < g_r([n]^d).
\end{align*}

By the above arguments, for any subset $A\subseteq [n]^d$ with $|A|\leq n^d-1$, we have $g_r(A) < g_r([n]^d).$
Therefore, among all subsets of $[n]^d$, $[n]^d$ is the unique subset admitting the maximum number of $r$-colorings without rainbow solutions to the $B_{k,h}$-equation.

\section{Concluding remarks}
\label{sec:conclu}

In this paper, we extend the Cameron-Erd\H{o}s problem to study rainbow solutions to generalized Sidon-equations in multidimensional grids.
We show that for integers $d, k, h, r, n$ with $d\geq 1$, $k\geq 2$, $h\geq 2$, $r\geq kh$ and $n$ sufficiently large,
\begin{itemize}
\item[(1)] if $A\subseteq [n]^d$ is a subset with $|A|\geq n^d-\frac{n^d}{\log n}$, then the asymptotic number of $r$-colorings of $A$ without rainbow solutions to the $B_{k,h}$-equation is $(1+o(1))\binom{r}{kh-1}(kh-1)^{|A|}$, and all but a $o(1)$ proportion of these $r$-colorings are $(kh-1)$-colorings;
\item[(2)] among all subsets of $[n]^d$, $[n]^d$ is the unique subset admitting the maximum number of $r$-colorings without rainbow solutions to the $B_{k,h}$-equation.
\end{itemize}
\vspace{-0.2cm}
We next provide two remarks regarding our result.
\vspace{0.2cm}

The first remark is that for any integer $k\geq 2$ and $k$ positive integers $h_1, h_2, \ldots, h_k$, we can also study the problem with respect to rainbow solutions to the equation $\sum_{i=1}^{h_1}x_{1,i}=\sum_{i=1}^{h_2}x_{2,i}=\cdots =\sum_{i=1}^{h_k}x_{k,i}.$
It is possible that for any subset $A\subseteq [n]^d$ with$|A|=(1-o(1))n^d$, the asymptotic number of $r$-colorings of $A$ without rainbow solutions to the above mentioned equation is $(1+o(1))\binom{r}{\left(\sum_{\ell=1}^{k}h_{\ell}\right)-1}\left(\left(\sum_{\ell=1}^{k}h_{\ell}\right)-1\right)^{|A|}$, and almost all of these $r$-colorings are $\left(\left(\sum_{\ell=1}^{k}h_{\ell}\right)-1\right)$-colorings.
However, the set $[n]^d$ may not attain the maximum number of $r$-colorings without rainbow solutions to the equation, among all subsets of $[n]^d$.
To see this, we consider the following two examples (for convenience, we only consider the case $k=2$).
Firstly, assume that $h_1<h_2$.
Let $A=\left\{\left\lceil\frac{h_1}{h_2}n\right\rceil+1, \ldots, n\right\}^d$.
Consider two arbitrary subsets $\{x_{1,1}, \ldots, x_{1,h_1}\}$ and $\{x_{2,1}, \ldots, x_{2,h_2}\}$ of $A$.
For any $j\in [d]$, the $j$th coordinate of $\sum_{i=1}^{h_1}x_{1,i}$ is at most $h_1 n$, and the $j$th coordinate of $\sum_{i=1}^{h_2}x_{2,i}$ is greater than $h_2\frac{h_1}{h_2}n=h_1 n$.
Hence, $A$ contains no solutions to the equation $\sum_{i=1}^{h_1}x_{1,i}=\sum_{i=1}^{h_2}x_{2,i}$, and thus the number of $r$-colorings of $A$ without rainbow solutions to the equation is $r^{|A|}\geq r^{(1-o(1))\left(\frac{h_2-h_1}{h_2}n\right)^d}$, which is larger than $(1+o(1))\binom{r}{h_1+h_2-1}(h_1+h_2-1)^{n^d}$ when $r$ is sufficiently large compared to $h_1$ and $h_2$.
Secondly, when $h_1$ is odd and $h_2$ is even, let $B=\left\{x\in [n]\colon\, \mbox{$x$ is odd}\right\}^d$.
Consider two arbitrary subsets $\{x_{1,1}, \ldots, x_{1,h_1}\}$ and $\{x_{2,1}, \ldots, x_{2,h_2}\}$ of $B$.
For any $j\in [d]$, the $j$th coordinate of $\sum_{i=1}^{h_1}x_{1,i}$ is odd, and the $j$th coordinate of $\sum_{i=1}^{h_2}x_{2,i}$ is even.
Hence, $B$ contains no solutions to the equation $\sum_{i=1}^{h_1}x_{1,i}=\sum_{i=1}^{h_2}x_{2,i}$, and thus the number of $r$-colorings of $B$ without rainbow solutions to the equation is $r^{|B|}\geq r^{\left(\frac{1}{2}n\right)^d}$, which is larger than $(1+o(1))\binom{r}{h_1+h_2-1}(h_1+h_2-1)^{n^d}$ when $r$ is large enough.
Therefore, when we consider general integers $h_1, h_2, \ldots, h_k$, the situation becomes more complicated.
One of the reason is that for $h_1=h_2=\cdots=h_k=h$, the largest subset of $[n]^d$ without solutions to the equation $\sum_{i=1}^{h}x_{1,i}=\sum_{i=1}^{h}x_{2,i}=\cdots =\sum_{i=1}^{h}x_{k,i}$ is small (for example, the largest Sidon set in $[n]$ has size $(1+o(1))\sqrt{n}$; see~\cite{ErTu941});
but when $h_1, h_2, \ldots, h_k$ are not all the same, there exist subsets of size $\Theta(n^d)$ without solutions to the equation $\sum_{i=1}^{h_1}x_{1,i}=\sum_{i=1}^{h_2}x_{2,i}=\cdots =\sum_{i=1}^{h_k}x_{k,i}$.

The second remark is from a geometric perspective.
In the case that $d\geq 2$, $k=2^{t-1}$ and $h=2$ for some $t \geq 2$, the equation $x_{1,1}+x_{1,2}=\cdots=x_{2^{t-1},1}+x_{2^{t-1},2}$ describes a (possibly degenerate) $t$-dimensional parallelotope (also called a $t$-dimensional parallelepiped) in the grid $[n]^d$.
The nondegenerate $t$-dimensional parallelotope $\mathbf{P}_{t}$ is defined as the polytope with points generated by a base point $x_0$ and $t$ linearly independent vectors $v_1, \ldots, v_{t}$ as follows: $$\left\{x_0+\sum_{i=1}^{t}\epsilon_iv_i\colon\, \mbox{$\epsilon_i\in \{0,1\}$ for all $i\in [t]$}\right\}.$$
Here, the linear independence of the vectors ensures that the parallelotope is of dimension $t$, and the condition $\epsilon_i\in \{0,1\}$ ensures that all pairs of opposite edges are parallel.
Note that $\mathbf{P}_{t}$ contains $2^{t}$ points,
and these points can be partitioned into $2^{t-1}$ pairs $(x_{1,1}, x_{1,2}), \ldots, (x_{2^{t-1},1}, x_{2^{t-1},2})$ such that $x_{\ell, 1}+x_{\ell, 2}=2O$ for all $\ell \in [2^{t-1}]$, where $O$ is the geometrical center of $\mathbf{P}_{t}$.
However, the single condition of $x_{1,1}+x_{1,2}=\cdots=x_{2^{t-1},1}+x_{2^{t-1},2}$ cannot guarantee the dimension and parallelism.
Therefore, a natural problem is to study the rainbow Cameron-Erd\H{o}s problem with respect to nondegenerate $t$-dimensional parallelotopes of multidimensional grids.
We can solve this problem for the case $t=2$ (i.e., for parallelograms) by modifying our proof slightly (in fact, we only need to modify the proofs of the counting results in Section~\ref{subsec:f}).
Note that in this case, the equation $x_{1,1}+x_{1,2}=x_{2,1}+x_{2,2}$ can guarantee the parallelism since we can choose $x_0=x_{1,1}$, $v_1=x_{2,1}-x_{1,1}$ and $v_2=x_{2,2}-x_{1,1}$ (so $x_{1,1}=x_0+0\cdot v_1+0\cdot v_2$, $x_{1,2}=x_0+1\cdot v_1+1\cdot v_2$, $x_{2,1}=x_0+1\cdot v_1+0\cdot v_2$ and $x_{2,2}=x_0+0\cdot v_1+1\cdot v_2$).
Hence, we only need to control the dimension, i.e., the four points cannot be collinear.
If the four points are collinear, then after choosing $x_{1,1}$ and $x_{1,2}$ in a subset $A\subseteq [n]^d$, there are at most $n$ choices of $x_{2,1}$ on the line passing through $x_{1,1}$ and $x_{1,2}$, and after choosing these three points, there is at most one choice of $x_{2,2}$ satisfying $x_{1,1}+x_{1,2}=x_{2,1}+x_{2,2}$.
Hence, there are at most $|A|^2 n$ degenerate parallelograms in $A$, and if we fix a point $v$, then there are at most $|A| n$ degenerate parallelograms containing $v$ in $A$.
Note that $|A|^2 n$ is much smaller than the lower bounds given by the Lemmas~\ref{le:fdkh} and \ref{le:fdkhj}, and $|A| n$ is much smaller than the lower bounds given by the Lemma~\ref{le:fAv} and Corollary~\ref{cor:fAv}.
Thus our proof still applies to nondegenerate parallelograms.
For the problem with a general $t>2$, some novel ideas are called for, since we also need to guarantee the parallelism.

\section*{Acknowledgement}

This paper is supported by the National Natural Science Foundation of China (Grant No. 12501492),
Shaanxi Fundamental Science Research Project for Mathematics and Physics (Grant No. 25JSQ043),
Shaanxi Province Postdoctoral Science Foundation (Grant No. 2024BSHSDZZ155),
and the Fundamental Research Funds for the Central Universities (Grant No. GK202506024).

\vspace{-0.2cm}
\section*{Declaration of competing interest}

There is no competing interest.
\vspace{-0.2cm}

\section*{Data availability}

No data was used for the research described in the article.

\begin{spacing}{0.8} 

\end{spacing}

\end{document}